\documentclass[11pt,a4paper]{article}

\usepackage{amsmath,amsthm,amssymb,color}
\usepackage[authoryear]{natbib}
\usepackage{booktabs}

\usepackage{bbm}
\usepackage{mathrsfs}
\usepackage[breaklinks=true]{hyperref}

\numberwithin{equation}{section}
\allowdisplaybreaks[4]

\theoremstyle{plain}
\newtheorem{theorem}{Theorem}[section]

\theoremstyle{definition}

\newtheorem{claim}{Claim}[section]

\newcommand{\E}{\mathbb E}
\newcommand{\p}{\mathbb P}
\def\P{\mathbb P}

\newcommand{\eq}{\eqref}

\def\be#1{\begin{equation*}#1\end{equation*}}
\def\ben#1{\begin{equation}#1\end{equation}}
\def\bes#1{\begin{equation*}\begin{split}#1\end{split}\end{equation*}}
\def\besn#1{\begin{equation}\begin{split}#1\end{split}\end{equation}}

\begin{document}

\title{\sc\bf\large\MakeUppercase{Segmentation and Estimation of Change-point Models: False Positive
Control and Confidence Regions}}
\author{\sc Xiao Fang, \sc Jian Li and \sc David Siegmund}
\date{\it The Chinese University of Hong Kong, Adobe Systems and Stanford University} 
\maketitle

\begin{abstract}
To segment a sequence of independent random variables at an unknown
number of change-points, we introduce new procedures that are based
on thresholding the likelihood ratio statistic. 
We also study confidence regions  based on the likelihood ratio statistic for the 
change-points and joint confidence regions for
the change-points and the parameter values.
Applications to segment array CGH data are discussed.

\end{abstract}

\medskip



\section{Introduction}
Diverse scientific applications have led to
recent interest in segmentation of models involving multiple 
change-points.  A model having some direct applicability and additional
theoretical interest for the insights it provides 
is as follows.  Let $X_1, X_2, \ldots, X_m$ be independent and normally
distributed with variances equal to 1.   Assume that there exist $M \geq 0$ 
and integers $0 = \tau_0 < \tau_1 < \ldots \tau_M < \tau_{M+1} = m$
such that the mean $\mu_i$ of $X_i$, $1\leq i\leq m$, is a step function with constant
values on each of the intervals $(\tau_{k-1}, \tau_{k}]$, $1\leq k\leq M+1$, but different values
on adjacent intervals.  
Segmentation amounts to determining the value of $M$, the $\tau_k$ and
perhaps also the $\mu_i$.
Because of the computational difficulty of sorting through all possible
partitions of $[1,m]$ to find the change-points when $m$ is large, 
there have often been
different algorithms for suggesting
a set of candidate change-points $\tau_k$ and for determining which of
those possible sets is ``correct.''  For example, one might use a
dynamic programming algorithm to propose a relatively small set of
possible $M$ and $\tau_k, 1 \leq k \leq M$, then use a statistical 
procedure to determine a final choice from those suggested in the
first stage of analysis.   
For recent reviews imbedded in otherwise original research articles
see \cite{Mu14} and
\cite{Fr14}.  Recent consistency results under essentially 
minimal conditions on the spacing and amplitude of the
change-points are given
in \cite{CC2017}.

Substantial motivation for recent research has 
been copy number variation (CNV) in genetics (e.g., \cite{OlVe04}, \cite{Pollack99},
\cite{Pi05}, \cite{La05}, \cite{Snijders03}, \cite{Zhao04}, \cite{ZhSi07},
\cite{NZ12}, \cite{Mu14},  \cite{ZYXS2016}). CNV can occur as somatic
mutations, especially in cancer cells, where they can involve a substantial portion of a
chromosome and show no particular pattern, or as germline mutation, 
which typically involves a short interval exhibiting an increase or decrease
in the mean followed by a decrease or increase that returns the mean to
a baseline value.
Like other 
genetic polymorphisms inherited CNV can be used to track relatedness 
of different individuals in populations or may be
of interest because of a possible relation to particular inherited diseases. 
Data in the literature can help us determine interesting sample sizes 
and values for parameters
in our numerical examples.  The sample size
$m$ is typically moderately large to large, while $M$ can be small or 
large in an absolute sense,
while still small compared to $m$; and consecutive change-points can be 
quite close together.

Another genomic application involves sequences of Bernoulli variables, which 
equal 0 or 1 according as the DNA letter at that 
location is A or T, or is C or G. Since a CG ``rich'' region is an indication of 
the presence of a gene or genes, it may 
be useful to segment a genome or part of 
a genome into regions of relatively low or high CG content. 
See, for example, \cite{Ch89} who used a Hidden Markov Model, 
or \cite{EGJ10}.
A variety of other examples motivated by particular scientific experiments
is given by \cite{DKK16}.  In particular they describe examples where
several consecutive changes are expected to have the same sign and where the pattern
of change-points may arise from a hidden Markov model. 

Scientific focus may emphasize detection and estimation of the
change-points, 
estimation of the step function of mean values, or a combination of the two.
Our primary focus is on the change-points themselves, 
which in a genomic context indicate
the existence and location of a signal of interest. 
 
To this end we study iterative thresholding methods that allow one
(with varying degrees of success, discussed below) to control the global 
false positive error rate; and subject to successful control, 
to understand the relative strengths and weaknesses of different methods. 
We also provide approximations to the local power (defined below) and
large sample joint confidence
regions for the change-points or for the change-points and mean values.

It bears emphasizing that we have {\sl not} considered a large class of other methods,
in particular
dynamic programming to compute
a penalized likelihood function, or two stage methods, where a 
list of candidate 
change-points obtained in the first stage is 
followed by a model selection method.
(Three of the four methods used below for comparative  purposes have been originally
proposed as two stage methods, but we have adapted them to be single stage 
thresholding methods.)
To the best of our knowledge, these methods all involve selection of 
arbitrary parameters that may have no simple statistical interpretation.
In contrast we have
produced a set of tools to study a restricted set of procedures in terms of the
classical statistical criteria of false positive and false negative error rates; and 
for each statistic there is a single thresholding parameter that we use to attempt to balance
these two error probabilities.

To motivate the methods introduced below
let $S_0=0$, $S_j = \sum_1^j X_i$ for $j\geq 1$ and
consider the generalized likelihood ratio
statistic for testing the hypothesis  $M = 0$ against
$M = 1$:  $\max_{0< j< m} |S_j - jS_m/m|/[j(1-j/m)]^{1/2}.$  This 
statistic is the basis
of the binary segmentation suggestion of \cite{Vo81}, 
which is a ``top down'' procedure, in the sense that one
tests all the data to determine if there is {\sl at least} one change-point
and iterates the procedure in the intervals immediately to 
the ``left'' and ``right'' 
of the most recently detected change-point.  We discuss below the
weaknesses of this method compared to a number of other thresholding
procedures.


Here we consider
``bottom up'' procedures motivated by the observation that in the presence of 
multiple change-points or to mitigate the effects of inadequately 
controlled drift in 
the ``baseline''  mean value (see below), it may be useful 
to compare a candidate
change-point at $j$ to an appropriate ``local'' background $(i,k)$,
where $i < j < k$.  
Similar approaches are the Wild Binary Segmentation (WBS) 
of \cite{Fr14}, which uses a random
set of possible backgrounds and an apparently empirically determined
threshold, and the method of
\cite{NZ12}, who use a limited number of symmetric backgrounds, to suggest
several sets of candidate change-points followed by model
selection to make the final choice. (Our suggested procedures 
could also form part of a
two-stage procedure, but here we consider in detail only a single stage,
which controls the rate of false positives.) 

To that end, consider 
\begin{equation} \label{maxLR}
\max_{0 \leq i < j< k \leq m} |Z_{i,j,k}|,
\end{equation}
 where for $i < j < k$:
\begin{equation} \label{LR}
Z_{i,j,k} = [S_j - S_i - (j-i)(S_k - S_i)/(k-i)]/[(j-i)(1-(j-i)/(k-i))]^{1/2}.
\end{equation}

Our first theoretical result is an approximation for the tail probability of
(\ref{maxLR}) when there is no change.  This approximation 
gives strong control on the
probability of a false positive result in the sense that if there are 
changes, say at $0=\tau_0<  \tau_1 < \ldots < \tau_M < \tau_{M+1} = m,$ 
the maximum of 
(\ref{LR})
over $n = 0, \ldots, M$ and $ \tau_n \leq i < j < k \leq \tau_{n+1}$ is stochastically smaller
than (\ref{maxLR}).  Hence, except for an event
of the probability evaluated asymptotically in Theorem 2.1,
any background interval 
$(i,k)$ where the statistic 
exceeds the threshold at an intermediate value of $j$ will contain 
at least one change-point.   

Our second principal result is an approximate likelihood ratio  confidence
region jointly for the change-points $\{ \tau_k, k = 1, \ldots M\}$
or for the change-points
and the mean values.  A related result allows one to approximate the
local power (cf. Section 3.2), which we find useful in helping 
us understand which change-points are
relatively easily detected and which may be missed.

In more detail, our segmentation procedure based on \eq{LR},
which we call the Local Likelihood Ratio (LLR)is as follows.
(In the following, we use the same acronym to refer to both the statistic and 
associated segmentation procedures.)
Because of local correlations between different $Z_{i,j,k}$,
thresholding (the absolute value of) (\ref{LR}) produces a
frequently large list of candidate
change-points $j$, each one against multiple backgrounds $(i,k)$.
Since our goal is to detect individual changes against 
the appropriate background 
we find it convenient in searching 
the list of candidates to require
that the background for one candidate
change-point $j$ not overlap another candidate change-point $j'$ in
the sense that if
$j < j'$,  the corresponding backgrounds should
satisfy $k \leq j'$ and $i' \geq j$.  This can be accomplished
by sequentially re-evaluating candidate change-points
until they satisfy the constraint.  
Hence, when a new change-point is identified, the existing putative change-points to its left and to 
its right may need to be removed or altered.
An approach
that requires very little  and usually no re-evaluation of candidate
change-points is to select the shortest of the possible backgrounds
$(i,k)$ from among those for which
$Z_{i,j,k}$ exceeds the required threshold, which is similar to the
method recommended in 
\cite{BCF16}. If there is a tie for the shortest
value of $k-i$, we choose the one with the largest value of $|Z_{i,j,k}|$.
Another possibile algorithm is
selection based on the largest value of the statistic, which
rarely leads to significant differences from selection based on the
shortest background, although now iteration to enforce the no overlap 
condition is common. 

For various scientific reasons and in particular
for determination of the confidence regions discussed below 
an important consideration is
the size, as well as the location, of the change.  
Although it seems natural to conjecture that the  
largest $|Z|$-value, subject to no overlap, provides the most
accurate estimate since it is based on 
a longer background, from simulations we can see
that this is by no means always the case. 
Since we have ``paid up front'' for
protection against false positive errors, we can also choose to 
look at a number of candidate
change-point-background combinations to find one that is
subjectively appealing.  Simple possibilities that seem to 
present themselves frequently
are to choose from  among the $(i,j,k)$ combinations having, say, the five 
shortest backgrounds the value of $j$ that appears most frequently, or  has
the largest $Z$-score, or also heads the list generated 
by the algorithm that 
is focused on the
largest $Z$-score.

We also study a pseudo-sequential
procedure (SLLR) where we initially set $i = 0,$ find the
smallest 
$k>i+1$
such that
$\max_{i < j < k} |Z_{i,j,k}|$ is above an appropriate threshold, set $j_1$ equal to
the largest such $j$ or the maximizing value of $j$, then set $i = j_1$ and 
iterate the process.  
This procedure has lower computational complexity than LLR, although it is still a bottom up 
procedure in the sense that each detected change-point is compared to a local background
that ideally contains no other change-points to introduce potential distortions into the 
location or size of the change.
Because SLLR has a lower threshold than LLR, it also has larger power to detect change-points.
However, as explained below Theorem \ref{t3}, we do not have as strong a theoretical guarantee 
that the false positive error probabilities are controlled.
See
the unpublished Stanford Ph. D. thesis of E.S. Venkatraman for an
early discussion of a similar idea.

The paper is organized as follows.  In Section 2 we 
give approximations to control the false positive probabilities of our
proposed and other segmentation methods.  
Approximate joint confidence regions are discussed in Section 3.   
Using simulations and analysis of some real data, we compare our 
methods in Section 4 with other thresholding 
methods that 
control false positive probabilities with varying degrees of success,
and we give some numerical examples for confidence regions.
Sections 5 contains extensions to exponential
families, and Section 6 some additional discussion. 
In the Appendix we prove the theorems stated in Sections 2 and 3.



\medskip\noindent
{\bf Remarks.}
(i) For some applications, e.g., for inherited
CNV, the signal to be detected extends over a
relatively short range in the form of a 
departure from a baseline value where one
change is followed by a second, nearby change in the opposite
direction. 
For problems of this form it seems
reasonable to use statistics adapted to the expected shape of the signals,
(e.g., \cite{OlVe04}, \cite{Mu14}, \cite{ZhLJS2010}).  
We give appropriate approximations for false positive control in Section 2
and consider such procedures in more detail in Section 4, where we show 
that they perform
very well even when there is no particular pattern to the change-points.
Inherited CNV also provide motivation for studying multivariate
observations (i.e., multiple DNA sequences), since the change-points 
may be difficult to detect in individual sequences and their 
occurrence in several sequences 
indicate possible relationships among those sharing the same 
change-points (\cite{ZhLJS2010}).

\smallskip\noindent
(ii) We have assumed the variance of the observations is known. The 
sequences are usually long enough that, under our assumption of independence, 
the variance can  be accurately estimated by
one-half the average of the squared
differences of consecutive observations.   This estimator avoids
the substantial upward bias of the empirical variance of the data,
which arises when there are multiple change-points
and the changes themselves show no particular pattern.
Alternatives are to use a function of the order statistics of
adjacent observations, e.g., the median of $|X_{i} - X_{i-1}|$ or 
the interquartile range of $X_{i+1} - X_i$, multiplied by suitable constants.
Still another estimator that may have some value is the average 
of squared second order differences of consecutive observations, which has
the virtue of nullifying the effect of linear drift and perhaps also the 
relatively slow oscillations that plague some genomic applications
(\cite{OlVe04}).
These estimators based on pairwise differences are inappropriate when 
the data are autocorrelated,
a problem we expect to study in the future.
An example where the empirical variance is satisfactory occurs when
the data are of the
form envisioned in Remark (i), where change-points occur in  a relatively sparse set 
of departures 
from a baseline value followed by a return to the baseline a few observations
later.

\smallskip\noindent
(iii)  In recent research some authors recommend 
a multiscale modification of  the likelihood ratio statistic.  One possibility is to modify
(\ref{LR}) by 
subtracting $\{2 \kappa \log[3m/\min(j-i,k-j)]\}^{1/2}$, 
$\kappa > 0$, in order to obtain greater power to 
detect relatively small changes that persist over longer intervals at the cost of less
power to detect large changes that come relatively close together.  
See, for example, \cite{DSp01} and \cite{Mu14}.  Our methods can be
adapted to study these modifications, and in Section 4 we
investigate a procedure based on the statistic recommended in \cite{Mu14}.  
However, these methods are not 
central to our studies for the following reasons.
(a) For problems of CNV detection, difficult detections in 
the synthetic data suggested
in the applied literature and in the real data in 
Section 4 often involve short intervals and relatively large changes.    
(b) What appear in the data to be small, relatively rare, changes 
may arise from technical artifacts in the
data and are not scientifically interesting 
(cf. \cite{OlVe04} and \cite{ZhLJS2010}).  
(c) Multiscale modifications are not uniquely defined; and in 
different problems different 
statistics may have slight advantages and disadvantages.  
(d) Multiscale methods do not appear to
adapt as naturally for the determination 
of confidence regions as the likelihood ratio statistic.

\section{Approximate $p$-values}

In what follows we write $A\asymp B$ to mean 
that $0<c_1\leq A/B\leq c_2<\infty$ for two absolute constants $c_1$ and $c_2$, 
and $A(b)\sim B(b)$ means $A(b)/B(b)\to 1$ as $b\to \infty$; also
$\varphi$ and $\Phi$ are the standard normal probability density function and distribution function,
respectively.

We have the following $p$-value approximation for
$\max_{0\leq i<j<k\leq m}|Z_{i,j,k}|$.
\begin{theorem}\label{t2}
Let $(X_1,\dots, X_m)$ be an independent sequence of normally distributed random variables  with
mean $\mu$ and variance 1.
Then for  $Z_{i,j,k}$ as defined by \eq{LR},
we have for $b\to \infty$ and $m\asymp b^2$,
\besn{\label{2.1}
&{\p} \{ \max_{0\leq i < j < k\leq m} |Z_{i,j,k}| \geq b \} \\
\sim & \frac{b^6 \Phi(-b)}{4} 
\sum_{u,v\in \{1,\dots, m\}: \atop u+v\leq m} 
\frac{(m-u - v)}{uv(u+v)} \nu[b (\frac{u}{v(u+v)})^{1/2}] \nu[b(\frac{v}{u(u+v)})^{1/2}]
\nu[b(\frac{u+v}{uv})^{1/2}].}
The function
$\nu$ is defined, e.g.,  in \cite{SiYa07} p. 112 and given to 
a simple approximation by the equation 
\[
\nu(x) = 
(\Phi(y)-1/2)/[y(y\Phi(y) + \varphi(y))],
\]
where $y = x/2$.
\end{theorem}

For the proof of Theorem \ref{t2} we use a new method beginning from an observation of \cite{ZhLi11}, which
was used there as the basis for
Monte Carlo simulation with a one (time) dimensional random
field and which we have used for an analytic approximation
involving  maxima of
certain three (or higher dimensional) random fields.  
The starting point is a number of simple observations, which 
require  a large number of detailed calculations for complete
justification.

Denote the right-hand side of \eq{2.1} by $p$. Rewrite $p$ as
\be{
p\sim \frac{\varphi(b) }{4b}
 \sum_{u,v\in \{1,\dots, m\}: \atop u+v\leq m}
(m-u - v)\left\{\frac{b^6}{uv(u+v)} \nu[b (\frac{u}{v(u+v)})^{1/2}] \nu[b(\frac{v}{u(u+v)})^{1/2}]
\nu[b(\frac{u+v}{uv})^{1/2}] \right\}.
}
It was shown in \cite{Si85} that $\nu(x)=\exp(-cx)+o(x^2)$ as $x\to 0$
for $c\approx 0.583$, while $x^2 \nu(x)/2\to 1$ as $x\to \infty$.
Therefore, the term inside the curly brackets  above is bounded. Hence
\ben{\label{A.1}
p\asymp b^5 \varphi(b)\to 0,
}
where we used the assumption that $m\asymp b^2$.

Fix a sufficiently small constant $c_0$.
We will prove first that
\besn{\label{2.1*}
&{\p} \{ \max_{0\leq i < j < k\leq m: \atop j-i, k-j\geq c_0 b^2} Z_{i,j,k} \geq b \} \\
\sim & \frac{1}{8} b^5 \varphi(b)
\sum_{u,v\in \{1,\dots, m\}: \atop u+v\leq m; u, v\geq c_0 b^2 }
\frac{(m-u - v)}{uv(u+v)} \nu[b (\frac{u}{v(u+v)})^{1/2}] \nu[b(\frac{v}{u(u+v)})^{1/2}]
\nu[b(\frac{u+v}{uv})^{1/2}].
}
We write
\bes{
&{\p}(\max_{0\leq i<j<k \leq m \atop  j-i, k-j \geq c_0 b^2} Z_{i,j,k}\geq b)\\
=& \sum_{0\leq i<j<k \leq m \atop  j-i, k-j \geq c_0 b^2} {\p}(Z_{i,j,k}\geq b, Z_{i,j,k}=\max_{0\leq r<s<t \leq m \atop s-r, t-s \geq c_0 b^2} Z_{r,s,t}) \\
= &\sum_{0\leq i<j<k \leq m \atop  j-i, k-j \geq c_0 b^2}  \int_0^{\infty}
{\p} (\max_{0\leq r<s<t \leq m \atop  s-r, t-s \geq c_0 b^2} Z_{r,s,t}\leq b+x | Z_{i,j,k}=b+x) {\p}(Z_{i,j,k} \in b+dx)\\
=&\sum_{C\log b\leq i<j<k \leq m-C\log b \atop  j-i, k-j \geq c_0 b^2}  \int_b^{b+1}
{\p}(\max_{0\leq r<s<t \leq m \atop  s-r, t-s \geq c_0 b^2} Z_{r,s,t}\leq x | Z_{i,j,k}=x) d{\p} + R
}
where $C$ is a positive constant to be chosen.  The rest of the proof involves a detailed analysis of these expressions
to show that $R$ and various other terms that have been ignored are indeed negligible.  These technical details are given
in Appendix A.  

\medskip\noindent
\noindent{\bf Remarks.}
(i)  
Other methods that appear to be adaptable to prove Theorem \ref{t2}, 
albeit with more, less intuitive,
computation are those of \cite{Si88a} and
of \cite{Ya13}.  

\smallskip\noindent
(ii) Usually we are interested in small probabilities, and then we can use
(\ref{2.1}) as given.  Occasionally we may be interested in
cases where $m$ is so large that the probability is not small.  In those
cases we can supplement our large deviation approximation with a ``Poisson'' 
approximation in the form
$ 1- \exp[-{\rm RHS}(\ref{2.1})]$, which reduces to our approximation 
when the probability is small.
See \cite{SiYa00} for a proof in a 
related case.

\smallskip\noindent
(iii)   Based on other, related, calculations (see, for example, (5)
in \cite{ZhLJS2010}), it seems clear that similar
results apply to  multivariate data with some (mostly minor)
changes.  This case is
particularly interesting for detection of inherited CNV, which are short and
sometimes difficult to detect in single DNA sequences (e.g.,
\cite{ZhLJS2010}).  It is also interesting to
infer which subsets of the distributions have changed at the various
change-points.
The required modifications of the approximation given in the theorem are   
replacing $\Phi(-b)$ by $f_d(b^2)$ where $f_d$ is the 
$\chi^2$ probability density function with $d$ degrees of freedom 
and (to account for the 
curvature of the sphere when the dimension $d$ is large) multiplying the entire expression by
$q^5$ and the arguments of the functions $\nu$ by
$q$, where $q = 1-(d-1)/b^2$.  See (2.9) below.

\smallskip\noindent
(iv) For a multiscale statistic along the lines suggested in Remark (iii) at 
the end of Section 1, where
we subtract, 
for example, $\{2 \kappa \log[3m/\min(j-i,k-j)]\}^{1/2}$ from (\ref{LR}), 
a similar approximation holds, with the right hand side modified by replacing
$b$ by $b(u,v) = b +\{2 \kappa \log[3m/\min(u,v)]\}^{1/2}$ and moving the 
expressions involving $b(u,v)$ inside the summation.  See (2.9) below for a
similar approximation involving the (\cite{Mu14}) recommended statistic.  

\smallskip\noindent
(v) In applications we may wish to put a lower and/or an upper bound on 
the length of the background, e.g., 
$m_0\leq j-i, k-j\leq m_1$. The appropriate change to \eq{2.1} is to restrict the summation on the right-hand side to $m_0\leq u,v\leq m_1$.
For applications where very short intervals between
change-points can occur, we may want to take $m_0 = 1.$
Values of $m_1 \ll m$ 
can be used to minimize detection of small jumps, which
may themselves reflect experimental artifacts that lead to drift in the
underlying distributions (cf. \cite{OlVe04}, \cite{ZhLJS2010}); and they 
speed up what may otherwise 
be time consuming computations
for large values of $m$.  

\smallskip\noindent
(vi)  If there are a large number of change-points to be detected, one 
might prefer to control the rate of false positive errors via the 
false discovery rate (FDR).  In change-point problems it is important to distinguish between
discoveries and ``tests,'' since in our context many correlated tests
may refer to relatively few change-points.  Efforts to clarify and deal with
this distinction that seem applicable in principle to 
our segmentation problem are \cite{SGA11}, \cite{HNZ13},  and \cite{SiYaZh11}; but since 
the number of change-points in our motivating examples is typically not
large, we do not consider this possibility in detail.

\smallskip\noindent
(vii)  One can choose to approximate the probability in the theorem 
by simulation; and to do this once for a particular study does not
seem to pose difficulties.
Simulation may also be  useful to study variations of our
problem under different models.  
But the analytic approximations are much faster to evaluate, 
and hence we find it
useful to perform
a relatively limited set of simulations to gain confidence that 
our approximations are reasonably accurate, then use the approximations to
study and compare different statistics under different conditions.
We will, nevertheless, see below that simulations always play an important role.

\smallskip\noindent
(viii)  The natural setting for these approximations is the
likelihood ratio statistic in exponential families, as
discussed briefly in Section 5.  A more robust, although often
quite conservative, approximation is to replace the discrete
time random walk of the theorem by continuous time Brownian
motion.  The resulting approximation would look the same,
but the functions $\nu$ would be replaced by 1, and the
sums would be integrals (perhaps still evaluated as sums).
While this would in principle allow the theorem to be applied to
a wide variety of statistics, in specific cases the approximation may be
quite conservative. 

\medskip

For the SLLR, we have a similar approximation to the
probability of a false positive detection. 

\begin{theorem}\label{t3}
Let $(X_1,\dots, X_m)$ be a sequence of independent  normally 
distributed random varibles with mean $\mu$ and variance 1.
Let $Z_{i,j,k}$ be defined as in \eq{LR}.
We have for $b\to \infty$ and $m\asymp b^2$,
\besn{\label{Seq}
&{\p} \{ \max_{0 < j < k \leq m} |Z_{0,j,k}| \geq b \} \\
\sim & \frac{1}{2} b^3 \varphi(b) 
\sum_{1 < k \leq m } \sum_{0 < j < k}
j^{-2}  \nu[b(((k-j)/(jk) )^{1/2}] \nu[b(k/(j(k-j)))^{1/2}].
}
\end{theorem}

Compared to the segmentation procedure using the LLR,   
the pseudo-sequential procedure has
the advantage that it is easier and faster to implement.
However it does not have as strong a theoretical guarantee that
the false positive error probabilities are controlled.
Moreover, since SLLR only uses a subset of the actual background of a 
change-point, 
the location of the change-point and the
magnitude of a change may not be as accurately estimated
as for LLR.
However, as simulations and examples below
suggest, SLLR is quite stable and efficient.

The statistic suggested by \cite{NZ12}  is
similar to LLR, but uses a background that is symmetric around
a putative change-point.  Consider
the local maxima with respect to $j$ of
\begin{equation}  \label{NZ}
Z_{j,h} = |[(S_{j + h} - S_j) - (S_j - S_{j - h})]|/(2h)^{1/2},
\end{equation}
where $h$ is a parameter to be chosen.
Since there is no obvious choice for $h$,
Niu and Zhang suggest maximizing
\eq{NZ} over a finite number of values of $h$.  For their applications to
copy number variation,  they suggest 3 values,
10, 20, and 30. To complete their method, which 
they call SaRa,  they use a model selection 
procedure following their use of \eq{NZ}, a step that we omit.

The methods of proof of Theorems 2.1 and 2.2  use the fact that
local perturbations of the processes around a high maximum,
in particular, the difference between a large value of $Z_{i,j,k}$ 
and values $Z_{i',j',k'}$ as a function of $i' \approx i$,
$j' \approx j$, and $k' \approx k$ involves a sum of three 
approximately independent
random walks.
Since the local random walks obtained from perturbations of $j$ and $h$
for \eq{NZ}  are not independent, we cannot apply the 
same methods to obtain a theoretical approximation to the false positive
error probability.  
The increments are weakly positively dependent, so it seems 
natural to conjecture 
that treating them as if they were independent would produce a 
slightly conservative approximation.

An approximation to the  
the maximum over $j$ and $h$ of \eq{NZ}, calculated on the assumption
that the local increments obtained from perturbations of $j$ and $h$
are independent is given by
\[{\p} \{ \max_{0 < t < m, \; 0 < h < \min(t,m-t) } |Z_{t,h}| \geq b \}
\]
\begin{equation} \label{NZalpha}
\sim 1.5 m b^3 \varphi(b) \sum_h \nu[b(3/h)^{1/2}] \nu[b (1/h)^{1/2} ]/h^2.
\end{equation}
Simulations indicate that \ref{NZalpha} is slightly conservative, as 
expected, so we have used simulated thresholds in comparing it to
other procedures in Section 4. 
Because of the restriction to a symmetric background, SaRa  
can suffer a serious loss of power when change-points are spaced 
irregularly, with some being close to others.  

Table 1 compares simulated values for the maximum of LLR 
with the approximation given in Theorem 2.1 
for various values of $m$
(i) when the maximum is constrained by $m_0\leq j-i, k-j\leq m_1$ and 
(ii) for the related Poisson approximation that is appropriate
when $m$ is so large compared to $b^2$ that the tail probability is 
not small. 
Some of the thresholds will be used in comparing different methods in Section 4.

\begin{table}[h]
\caption{Approximation \eq{2.1}.
Simulated values
based on $N= 10000$ repetitions in the first three rows, 1000 in the last two 
rows and 2000 otherwise.
}
\begin{center}
\begin{tabular}{c|c|c|c|c|c}
\hline
$b$ & m & $m_0$ & $m_1$ & $p_{\rm Approx}$ &Monte Carlo  \\\hline
3.64 & 25 & 1 & 24 & 0.050 & 0.052 \\
4.00 & 50 & 1 & 49 & 0.050 & 0.049 \\
4.30 & 100 & 1 & 99 & 0.049 & 0.046 \\
4.54 & 200 &1 &199 & 0.049 & 0.049 \\
4.68 & 300 & 1 & 299 & 0.048 & 0.049 \\
4.76 & 400 & 1 & 399 & 0.049 & 0.048  \\
4.83 & 500 & 1 & 499 & 0.049 & 0.047 \\
4.83 & 500 & 1 & 100 & 0.043 & 0.042 \\
4.83 & 500 & 1 & 50 & 0.034 & 0.035 \\
4.71 & 500 & 1 & 50 & 0.056 & 0.053  \\
4.60 & 500 & 1 & 100& 0.109 & 0.103  \\
4.77 & 500 & 1 & 100 & 0.056 & 0.054  \\
4.71 & 500 & 3 & 100 & 0.054 & 0.043 \\
4.45 & 500 & 3 & 50 & 0.117& 0.108 \\
5.17 & 2000 & 1 & 1000 & 0.054 & 0.042 \\
4.99 & 1000 & 1 & 300 & 0.053 & 0.046 \\
4.40 & 1000 & 1 & 300 & 0.45& 0.41 \\
4.30 & 1000 & 1 & 300 & 0.58 &  0.51     \\ \hline
        \end{tabular}

        \label{table:DD}
\end{center}
\end{table}

Some numerical experimentation, not reported here in detail,
suggests that the approximation of Theorem 2.2 is also reasonably accurate.
For example, for  $m = 500$, the threshold $b = 4.34$
yields the probability 0.051, while
simulations (2500) repetitions give the probability 0.045.

In Section 4 we compare the methods described above with 
the widely applied method of 
\cite{OlVe04} and a threshold version of 
the method of \cite{Mu14}.  For completeness we give 
appropriate  approximations for their false positive control. 
Consider
\begin{equation} \label{eq:lr0}
\max_{0 \leq j < j+n \leq m} Z_{j,n},
\end{equation}
where
\begin{equation} \label{eq:lr2}
Z_{j,n} = \frac{|S_{j+n} - S_j - nS_m/m|}{[n(1-n/m)]^{1/2}} - \{2 \kappa \log[3m/n(1-n/m)]\}^{1/2}.
\end{equation}
The case $\kappa = 0$ is 
called CBS (circular binary segmentation) and
was suggested in \cite{OlVe04}, where it was applied to
copy number data. It is the likelihood ratio statistic for the
case that there exists a pair of changes, where the second change
is equal in magnitude but opposite in sign to the first change. 
The case $\kappa = 1$ is the multiscale statistic of \cite{Mu14}
(which we call Multi),  
who argued that
the CBS puts relatively too much power into
the detection of
short intervals of large amplitude at the cost
of considerably less power to detect
relatively long intervals of small amplitude.

An approximation for the false positive probability of \eq{eq:lr2}
stated here for the case of  $d$-dimensional $X_i$ 
with covariance matrix $\Sigma$ (cf. Remark (iii) following the
statement of Theorem 2.1), is given by
\besn{\label{2*}
&{\p} \Big\{ \max_{0\leq j<j +n\leq m \atop m_0\leq n\leq m_1}\big\{ \frac{|\Sigma^{-1/2}(S_{j+n} - S_j - nS_m/m)|}{[n(1-n/m)]^{1/2}} - \{2 \kappa \log[3m/n(1-n/m)]\}^{1/2} \Big\}\geq b \big\}\\
&\sim 2 \sum_{n=m_0}^{m_1}(m-n) f_d\big(b_n^2\big) \left( \frac{b_n^4q_n^3}{(2n(1-n/m))^2}\right)
\nu^2\left(\frac{b_n q_n}{[n(1-n/m)]^{1/2}} \right),
}
where $f_d$  denotes the chi-square probability density function
with $d$ degrees of freedom,
$b_n=b+\{2 \kappa \log[3m/n(1-n/m)]\}^{1/2}$,
and $q_n=1-(d-1)/b_n^2$.
The derivation of \eq{2*} is similar to that of Theorem 2.1 for $d=1$,
modified as suggested in the proof of (5) of \cite{ZhLJS2010} for $d>1$.

The method of proof of Theorem 2.1 appears to be applicable 
to some sparse interval systems considered in the literature, although
theoretical or numerical justification for the approximations 
we can formally obtain requires investigation.
For example, consider the sparse interval system (2.3) of Chan and 
Chen (2017) with $T=m, h=m, r=1+\epsilon$ with small $\epsilon>0$.
A modification of the calculations used in the proof of our Theorem 2.1 
(given in Appendix A) produces the approximation
\bes{
&\P \{\max_{0\leq i<j<k\leq m \atop j-i, k-j\in \{\lfloor (1+\epsilon)^a \rfloor: a\in \mathbb{Z}^+\}} |Z_{i,j,k}|\geq b\} \\
\sim & \frac{b^6 \Phi(-b)}{4}\sum_{u,v\in \{1,\dots, m\}, u+v\leq m \atop u,v \in \{\lfloor (1+\epsilon)^a \rfloor: a\in \mathbb{Z}^+\}}
(m-u-v)\prod_{i=1}^3 (4d_i^2)\nu(2d_i),
}
where
\be{
d_1=\frac{b}{2} \Big[\frac{(\lfloor \epsilon u\rfloor \vee 1)v}{u(u+v)}\Big]^{1/2},
}
\be{
d_2=
\begin{cases}
\frac{b}{2}\Big[\frac{1}{u}+\frac{1}{v}\Big]^{1/2} & \text{if}\  (\lfloor \epsilon v\rfloor \vee 1)=(\lfloor \epsilon u\rfloor \vee 1)=1\\
\frac{b}{2}\Big[  \frac{2(u^2+v^2+uv)}{uv(u+v)}   \Big]^{1/2} & \text{otherwise},
\end{cases}
}
\be{
d_3=\frac{b}{2} \Big[\frac{(\lfloor \epsilon v\rfloor \vee 1)u}{v(u+v)}\Big]^{1/2}.
}
Here $d_1$ and $d_3$ can be interpreted as the standardized drifts of the 
local random walks obtained by perturbing $i$ and $k$, respectively,
by the amounts of $(\lfloor \epsilon u\rfloor \vee 1)\times \mathbb{Z}^+$,
while $d_2$ corresponds to perturbing $j$ in one case, and shifting 
$(i,j,k)$ in the other case.
For $m=500, b=4.83, \epsilon=0.1$, this approximation 
gives $0.041$ and simulation based on 2000 
repetitions gives 0.040.
For $m=1000, b=5.1, \epsilon=0.1$, the approximation gives $0.028$ and simulation 
based on 2000 repetitions give $0.023$.
Similar approximations can also be obtained if we allow the threshold $b$ to depend on $i,j,k$ as in Chan and Chen (2017).

\section{Confidence Regions and Local Power}

We continue to assume independent normal observations
$X_1,\dots, X_m$ with mean
values forming a step function with jumps at $\tau_k, \; 1 \leq k \leq M$ and 
variance equal to one.  
For a given value of $M$, we can use the likelihood ratio statistic to construct a 
joint confidence region for the change-points $\tau=(\tau_1,\dots,\tau_M)$  or
for the change-points and mean values $\mu=(\mu_1,\dots, \mu_{M+1})$.

We use the inverse relation between confidence intervals and hypothesis tests.
For testing a putative value of the positions of change-points and the corresponding mean 
values, the maximum log likelihood ratio statistic is
\begin{equation} \label{1}
\begin{split}
T_{\tau,\mu}=&\max_{0< t_1<\dots <t_M<m} \sum_{k=1}^{M+1} \frac{(S_{t_k}-S_{t_{k-1}})^2}{2(t_k-t_{k-1})}
-\sum_{k=1}^{M+1} \left[\mu_k(S_{\tau_k}-S_{\tau_{k-1}})-\frac{\mu_k^2}{2}(\tau_k-\tau_{k-1})\right]\\
=:&
\max_{0< t_1<\dots <t_M<m} U_{t,\tau,\mu},
\end{split}
\end{equation}
$t=(t_1,\dots, t_M)$, $t_0=\tau_0=0, t_{M+1}=\tau_{M+1}=m$, 
and $S_i=\sum_{j=1}^i X_j$ for $0\leq i\leq m$.
The $1-\alpha$ confidence region consists of those $\tau$ and $\mu$ 
such that $T_{\tau,\mu}\leq a_{\tau,\mu}$ where
\begin{equation} \label{2}
\p_{\tau,\mu}(T_{\tau,\mu}>a_{\tau,\mu})=\alpha.
\end{equation}
If we are only interested in the confidence region of $\tau$ and treat $\mu$ as a nuisance parameter, the maximum log likelihood ratio statistic is
\begin{equation}\label{3}
T_\tau=\max_{t_1,\dots, t_M} \sum_{k=1}^{M+1} \frac{(S_{t_k}-S_{t_{k-1}})^2}{2(t_k-t_{k-1})}
-\sum_{k=1}^{M+1} \frac{(S_{\tau_k}-S_{\tau_{k-1}})^2}{2(\tau_k-\tau_{k-1})}.
\end{equation}
By sufficiency the conditional distribution of $T_\tau$ given $\{S_{\tau_k}: 1\leq k\leq M+1\}$ does not depend 
on $\mu$. Therefore, a $1-\alpha$ confidence set for the change-points is 
the set of $\tau$ such that $T_\tau\leq a_{\tau, S_{\tau_1},\dots, S_{\tau_M}}$ where
\begin{equation}\label{4}
\p_{\tau}(T_\tau> a_{\tau, S_{\tau_1},\dots, S_{\tau_M}}|\tau, S_{\tau_1},\dots, S_{\tau_M})=\alpha.
\end{equation}

In the case there is known to be only one change-point, i.e., $M=1$, for 
exponentially distributed random 
variables, the exact value of
the left-hand side of (\ref{4}) was given by \cite{Wo86}. 
For $M=1$, asymptotic approximations for the left-hand side of both
(\ref{2}) and (\ref{4}) were given by \cite{Si88} for 
distributions from exponential families.
Since the asymptotic approximations in \cite{Si88} seem difficult to generalize to the case where $M\geq 2$, here we use a different approach to obtain asymptotic approximations for the left-hand side of both (\ref{2}) and
(\ref{4}) for $M\geq 1$.


\subsection{Tail approximations}

To construct the joint confidence region for the change-points and the
corresponding parameters, for each $\tau$ and $\mu$, we need to find $a_{\tau,\mu}$ such
that $$\p_{\tau, \mu} (T_{\tau,\mu}>a_{\tau,\mu})=\alpha$$ 
where $T_{\tau,\mu}$ is defined in (\ref{1}). The following theorem, the
proof of which is deferred to Appendix B,  
gives an approximation to
$$\p_{\tau, \mu} (T_{\tau,\mu}>a)$$
for large $a$.
We assume that the putative change-points are close enough to the 
true change-points that the maximum can be taken over relatively 
small neighborhoods ($|t_k-\tau_k|\leq n_k$) of the putative 
change-points, i.e.,
$$\p_{\tau, \mu} (T_{\tau,\mu}>a)\sim \p_{\tau, \mu}(\max_{t: |t_k-\tau_k|\leq n_k}U_{t,\tau,\mu}>a).$$
The assumption \eq{5} imposed in the theorem also 
ensure that the change-points are 
reasonably well separated from one another.
Despite these technical assumptions, simulation shows that our approximation is reasonably accurate (cf. Tables \ref{table9} and \ref{table10}).

\begin{theorem}\label{t4}
Let $\tau=\{\tau_1,\dots, \tau_M\}$ and $\mu=\{\mu_1,\dots, \mu_{M+1}\}$ be
defined as above.
Define $\delta_k=\mu_{k+1}-\mu_k$ for $1\leq k\leq M$ and $m_k=\tau_k-\tau_{k-1}$ for $1\leq k\leq M+1$.
Suppose that $|\delta_k|\asymp 1$ and
\begin{equation}\label{5}
1\ll a \ll n_k \ll (m_k \wedge m_{k+1})/a,
\end{equation}
where $A\ll B$ means $A/B\to 0$.
We have
\begin{equation} \label{6}
\p_{\tau, \mu}(\max_{t: |t_k-\tau_k|\leq n_k}U_{t,\tau,\mu}>a)\sim \p(\sum_{k=1}^M W_k+\frac{1}{2}\chi^2_{M+1}>a),
\end{equation}
where $U_{t,\tau,\mu}$ was defined in \eq{1}, $W_1,\dots,W_M, \chi^2_{M+1}$ are independent, $\chi^2_{M+1}$ is a chi-squared
random variable with $M+1$ degrees of freedom, and for $1 \leq k \leq M$ the
distribution of $W_k$ is given by 
\begin{equation}\label{51}
\p(W_k> x)=2\nu(|\delta_k|)e^{-x}-\nu^2(|\delta_k|)e^{-2x}, \ \forall \ x\geq 0
\end{equation}
for $1\leq k\leq M$.
\end{theorem}

We have a similar approximation for the left-hand side of \eq{4}.

\begin{theorem}\label{t5}
Let $T_{\tau}'$ be defined as in (\ref{3}) with the maximum taken over $|t_k-\tau_k|\leq n_k$ for $1\leq k\leq M$.
Define $\hat{\delta}_k=\hat{\mu}_{k+1}-\hat{\mu}_{k}$ for $1\leq k\leq M$, $\hat{\mu}_k=(s_{\tau_k}-s_{\tau_{k-1}})/(\tau_k-\tau_{k-1})$ and $m_k=\tau_k-\tau_{k-1}$ for $1\leq k\leq M+1$.
Suppose that $|\hat{\delta}_k|\asymp 1$ and
\begin{equation*}
1\ll a \ll n_k \ll (m_k \wedge m_{k+1})/a.
\end{equation*}
We have
\begin{equation}\label{201}
\p_\tau(T_\tau'> a|S_{\tau_1}=s_{\tau_1},\dots, S_{\tau_M}=s_{\tau_M})\sim \p(\sum_{k=1}^M W_k>a)
\end{equation}
where $W_1,\dots,W_M$ are independent and have the
same distributions as in Theorem \ref{t4} with $\delta_k$ replaced by $\hat{\delta}_k$.
\end{theorem}


It is easy to evaluate the distributions of $\sum W_k$ and $\sum W_k + \chi^2_{M+1}$ by
Fourier inversion, for values of $M$ up to about 100, and by asymptotic methods
in the rare case that still larger values of $M$ are of interest.  
We start from the standard inversion formula for a probability density $f$
function with an integrable 
characteristic function $\hat{f}$:
\[ f(x) = (2 \pi)^{-1} \int_{- \infty}^{\infty} \exp(-\sqrt{-1} \lambda x) \hat{f}(\lambda) 
d\lambda.
\]
For a distributions on the non-negative numbers we integrate
this from 0 to $b$ to find that probability to the left of $a$ equals
\[ \pi^{-1} \int_0^{\infty} {\bf Re}\{[1-\exp(-\sqrt{-1} \lambda a)] \hat{f}(\lambda)\} d\lambda/\lambda. 
\]
For our special case, for simplicity assume that
$\delta_k = \delta$ for all $k$.   Let $\nu = \nu(\delta)$ and $\hat{f} (\lambda) =
(1-\nu)^2 + 2 \nu/(1+ \sqrt{-1}\lambda) -2 \nu^2/(2 - \sqrt{-1}\lambda)$ denote the
characteristic function of $W_k$.  Let $\hat{g}(\lambda)$ be the characteristic
function of a $\chi^2_{M+1}$ random variable.  Finally, let
$h(\lambda) = \hat{f}^m(\lambda)*g(\lambda)[1-\exp(\sqrt{-1} \lambda a)]
/(1+\sqrt{-1} \lambda)$.
Then the probability on the right hand side of  (\ref{6}) equals
$ 1 - \int_0^{\infty} {\rm \bf Re}[ h(\lambda)] d\lambda/\pi.$
For Theorem 3.2 a similar expression without the
factor $h$ provides a numerical value for the
approximation.  

In Tables \ref{table9} and \ref{table101} we use simulations to check the accuracy of the approximations
of Theorems 3.1 and 3.2, respectively.
The number of change-points is $M=2$. The 
other parameters are indicated in the tables.
For different values of $\delta_1$ and $\delta_2$, we compute the threshold $a$
such that our approximation of the relevant probability equals 0.05.
The values $\hat{p}$ denotes a Monte Carlo estimate of the 
appropriate probability with $n_1=n_2=m$, based on
10000 repetitions each. We see that the
approximations are  reasonably accurate for the range $1 <|\delta|<2$.

\begin{table}[h]
\caption{Approximation \eq{6} for $M=2$. 
Simulated values based on $N=10000$ repetitions.
The values of $\hat{p}$ in parentheses correspond to those $\delta_2$ in parentheses. }
\label{table9}
\begin{center}
\begin{tabular}{c|c|c|c|c}
\hline
$m/\tau_1/\tau_2$ & $\delta_1$ & $\delta_2$ &  $a$ & $\hat{p}$ (Monte Carlo) \\
\hline
105/35/70 & 1.5 & 1.5\; (-1.5) & 6.55  & 0.052\; (0.055) \\
&2 & 2\; (-2) & 6.11 & 0.048 \; (0.043) \\
&2.25  & 2.25\; (-2.25) & 5.92  & 0.058\; (0.042) \\
&1.5   & 0.75 \; (-0.75) & 6.94  & 0.060\; (0.062) \\ \hline
210/70/140 & 0.75 & 0.75\; (-0.75) & 7.33 & 0.059\; (0.067)\\
& 1.5 & 1.5\; (-1.5) & 6.55 & 0.051\; (0.048) \\
& 2 & 2\; (-2) & 6.11 & 0.044 \; (0.046)\\
& 1.5 & 0.75\; (-0.75) & 6.94 & 0.055\; (0.060)\\ \hline
\end{tabular}
\end{center}
\end{table}

\begin{table}[h]
\caption{Approximation \eq{201} for $M=2$. 
Simulated values based on $N=10000$ repetitions.
The values of $\hat{p}$ in parentheses correspond to those $\hat{\delta}_2$ in parentheses.}
\label{table101}
\begin{center}
\begin{tabular}{c|c|c|c|c}
\hline
$m/\tau_1/\tau_2$ & $\hat{\delta}_1$ & $\hat{\delta}_2$ &  $a$ & $\hat{p}$ (Monte Carlo) \\
\hline
105/35/70 & 1.5  & 1.5\; (-1.5)  & 4.28 & 0.054\; (0.047) \\
& 2 & 2\; (-2) & 3.80 & 0.045\; (0.047)  \\
& 2.25 & 2.25 \;(-2.25) & 3.59 & 0.043\; (0.042)\\
& 1.5 & 0.75\; (-0.75) & 4.68 & 0.053\; (0.056)\\ \hline
210/70/140 & 0.75 & 0.75 \;(-0.75) & 5.09 & 0.055 \;(0.057)\\
& 1.5 & 1.5\;(-1.5) & 4.28  & 0.049 \;(0.047)  \\
& 2 & 2\; (-2) & 3.80 & 0.049 \; (0.047) \\
& 1.5 & 0.75 \;(-0.75) & 4.68  & 0.056 \;(0.055) \\ \hline
\end{tabular}
\end{center}
\end{table}

For a simple example of a confidence region for the
change-points, we simulated $m = 161$ observations with changes
in the mean value of size $\pm 2$ at observations 51, 91, and 121.  In
the first simulation $\hat{\delta} \approx 2$ for
all three change-points.  This value gave a threshold of
4.95 for a 95\% conditional confidence region.  The 
joint confidence region  consisted of 
the point estimators
51, 91, and either of 121 or 122.  In a second simulation with the
same parameters, the 
smallest estimate of $\hat{\delta}$ was 1.5, which 
if used for all three change-points would lead to a conservative threshold,
in this case equal to  5.6. 
The joint confidence region based on this threshold was substantially larger.
The union of the three regions was 
50, 51, 91, 92, 93, 121,122.  The joint confidence region consisted of
7 of the $2 \times  3 \times 2 = 12$ possible combinations of these 
values;  we omit the details.
When the size of the changes was decreased to $\pm 1.5$, we again used
the smallest value of $\hat{\delta}$, which again gave a 
threshold of 
5.6, and the 95\% joint confidence region extended up to 5 observations
away from the change-points at 51 and 91, and a couple of observations 
away from  121.  As a reflection of the fluctuations in the 
sample paths of the random walk, the regions around the individual
change-points were neither symmetric nor connected.

For applications to copy number variation, see Section 4.2.

\medskip
\noindent{\bf Remark.}  As one sees from an examination of the
conditions of Theorem 3.1 and Theorem 3.2, the methods discussed in this 
section work well if
the sizes of the changes and the distances between them are reasonably large.
If there is a mixture of large and small changes, or if it is unclear whether
a putative change is real or not, the procedure can be adapted appropriately.
For example, suppose we are interested in a joint confidence region for
the change-points when there are clear changes close to $\tau_1 < \tau_3$,
with what may or may not be a change at $\tau_2 \in (\tau_1, \tau_3)$.  
In taking the maximum indicated above, one can fix the value  $t_2 = \tau_2$ 
and maximize only over $t_1$ and $t_3$, while evaluating the conditional
probabilities given by all three $\tau_i$.  Whether there is a change at
$\tau_2$ or not, the conditional probability adapted from Theorem 3.2 now 
involves the sum of two conditionally independent maxima, not three.  
To be more 
conservative in protecting against a change-point near, but not exactly at,
$\tau_2$, we can bracket $\tau_2$ by, say $\tau_{20} < \tau_2 < \tau_{21}$,
and proceed from there.  The confidence coefficient is still asymptotically 
as given in Theorem 3.2, but the confidence region itself may have 
changed due to the change in the statistics used for conditioning.  
Presumably unnecessary conditioning  
leads to 
less accurate estimation.

\subsection{Power}
To help our intuition concerning the relation between background and
size of  a change that makes a particular change-point either easy or 
difficult to detect and to compare different procedures under hypothesized
conditions, it is helpful to have an approximation for the
power to detect a change.  

When the size of a change in the mean value is $\delta >0$ and the 
(largest possible) background
is $(i^*,k^*)$ for a change-point at $j^*$, we define 
the {\it marginal power} to be
\ben{\label{st1}
1-\Phi(b-\delta [h_1h_2/(h_1+h_2)]^{1/2}),
}
where $h_1 = j^* - i^*, \; h_2 = k^* - i^*.$
This is just the marginal probability that the statistic $Z_{i,j,k}$ evaluated
at the true change-point $j=j^*$ with the largest possible background $i=i^*, k=k^*$ exceeds the threshold $b$.
A detection may fail to occur at $i^*, j^*, k^*$, but occur 
at $i',j',k'$ which is a local perturbation of the
values $i^*, j^*, k^*$ in the sense that $i^*\leq i'<j'<k'\leq k^*$.
Using a similar argument as in the derivation of \eq{6} (cf. Appendix B),
we can approximate the probability of such a detection by
conditioning on $Z^2_{i^*, j^*, k^*}$ to obtain
\ben{\label{st2}
2 \int_0^{b^2/2} \p\{\sum_{i\in \{0,1,2\}} W_i > b^2/2 - x\} f(2x;1,\lambda) dx, 
}
where
$f(\cdot; 1,\lambda)$ is the probability density function of a $\chi^2$ distribution with one
degree of freedom and noncentrality parameter $\lambda = \delta^2 h_1h_2/(h_1 + h_2)$,
 $W_0, W_1, W_2$ are independent, $W_0$ is nonnegative and has the probability distribution
$\p\{ W_0 > x\} =  2 \nu(\Delta) \exp(-x) - \nu^2(\Delta) \exp(-2x)$ for $x \geq 0$ with 
$\Delta=b\sqrt{1/h_1+1/h_2}$,
and for $ i = 1,2$, $W_i$ is nonnegative and has the distribution given by
$\p\{ W_i > x\} =   \nu(\Delta_i) \exp(-x)$ for $x\geq 0$ with $\Delta_1 = \Delta/(1+h_1/h_2)$
and $\Delta_2 = \Delta/(1+h_2/h_1)$.
We use the term {\it local power}  
to denote the sum of the
marginal power \eq{st1} and the perturbation \eq{st2}.
Similar approximations can be obtained for the pseudo-sequential procedure, for
multidimensional statistics, and for multiscale statistics.
We omit the details.

\section{Simulations and Applications}

In this section we report the result of numerical exercises involving simulated and
real data to compare a number of
different segmentation procedures, with emphasis on their efficiency 
to detect change-points without an excessive number of  false positive errors.
We consider only thresholding algorithms that control the
false positive error rate under the global null hypothesis that there
are no change-points.  As we see below on the basis of
simulations that control is 
compromised to varying degrees
when iteration to find multiple change-points is required.

In contrast to LLR, SLLR and SaRa, 
both CBS and Multi are ``top down'' procedures, where
we begin by searching the entire interval of observations.
When one change-point (respectively, a
pair of change-points) is detected, the interval searched is divided into
two (respectively, three subintervals), and those subintervals are searched
for additional change-points.  
Since the methods are designed to detect change-points occurring in pairs,
under various conditions, e.g., when
there is only one change-point to be detected in a search interval,
or when consecutive changes are both positive or both negative,
one of the paired 
``detections'' often suggests a change-point very near to one end-point of
the search interval.  This is usually a false detection that is   
easy to recognize and disregard,
although the decision to disregard it has an element of subjectivity.  
To minimize
this subjectivity in our simulations, after some experimentation
we usually discard any detection having a distance to an end-point of
the interval searched that is within 5\% of the length of that 
interval. If both detections are within this distance, the one closer to
an end-point is discarded.  If they are equally distant from an end-point,
the one to be discarded is chosen at random.  While objective, this rule can
in some cases lead to 
errors, so in practice  we recommend making a subjective decision
based on a careful examination of the data.

Although the top down iterations of CBS and Multi make it 
natural to suspect that their false positive error control
may be inadequate, in most cases
this does not appear to be a 
major problem.
If a large interval is partitioned into smaller intervals by 
correctly detected change-points,
the false positive probability for CBS for the initial
interval is numerically very close to the sum of the
probabilities for the subintervals, so the sum of
the false positive probabilities for the small intervals is roughly the
same as that of the initial search.  For Multi, this sum is much less than
the false positive probability of the initial search (provided the 
value of $m$ is used for all searches, not changed to reflect the 
lengths of the different subintervals).
It appears that for both of these statistics the
main source of false positive errors arises, fortunately not often,  when  
a correct detection is paired with a false
detection that is not close enough to an endpoint to be excluded. 

It is also possible to give approximations for the local power for 
these two statistics, at least under the simplifying conditions that
we are at a stage of the search where there is only one or a pair of
change-points to be detected in the interval searched. 
For simplicity we consider only the CBS statistic
when either (i) the mean before the first change-point 
at $\tau_1$ equals $\mu_1$, between the first and second change-point
at $\tau_2$  equals $\mu_2$,
and returns to the value $\mu_1$ after $\tau_2$, or (ii) 
there is only one change-point at $\tau_1$ and $\tau_2 = m$.
Denote the magnitude of the change by $\delta$
and let $n_0 = \tau_2 - \tau_1$ denote the length of the changed interval.
Approximations and some calculus similar to that given in Section 3.2 lead to
\besn{\label{202}
&\p_\delta \left\{ \max_{0\leq i<j\leq m}  \frac{|S_j-S_i - (j-i)S_m/m|}{[(j-i)
(m-j+i)/m]^{1/2}} \geq b  \right\}\\
\approx & \Phi ( \delta [n_0(1-n_0/m)]^{1/2}-b )
+2\int_0^{b^2/2} \p(W_3+W_4\geq b^2/2-x) f(2x;1,\delta^2 n_0(1-n_0/m)) dx,
}
where $f(\cdot; 1, \lambda)$ is the density function of the chi-squared distribution with 1
degree of freedom  and noncentrality parameter $\lambda$, and
$W_3, W_4$ are independent  
nonnegative random variables similar to those appearing 
in the approximation for the local power of LLR
(cf. \ref{st2}).  If $\tau_2 < m$, both have
the distribution given by 
$\P(W_3>x)=2\nu (\Delta) \exp(-x)-\nu^2(\Delta)\exp(-2x)$
for $x \geq 0$, where $\Delta=b/\sqrt{n_0(1-n_0/m)}$.  If
$\tau_2 = m$, the right hand tail of the
distribution of  $W_4$ equals $\Delta \exp(-x)$.
Similar results
hold for the power of Multi and of SLLR.  

We do not consider in detail other 
top down iterative thresholding procedures that appear to have poorly 
controlled false positive
error rates if iterated with the same threshold and poor power under 
easily understood conditions.
One is the classical binary segmentation procedure of
\cite{Vo81}, which has a false positive  error probability that builds up
very quickly with the number of iterations required.  For example, suppose that
we use the  threshold of $b = 3.0$, which for $m = 300$ gives a global 
false positive probability
of approximately 0.05 on the initial search, and assume that we correctly detect the
four change-points of the first example in Table \ref{tableexamples}.  
Then the sum of the
false positive probabilities searching for a fifth change-point in the 
intervals between those already detected is about 0.126, 
and it would be larger if more iterations are required.  
There are various ways to adjust the thresholds to maintain control
of the false positive probabilities,
but this statistic has very poor power under conditions where
other statistics have no problems.  If
there is an increase (decrease) in the mean 
value followed by a decrease (increase) of a similar magnitude.
the statistic tends to average the two changes and detect neither.
For similar reasons we also
have omitted the thresholding procedure 
suggested by Aston and Kirch (2012), which
is similar to CBS and Multi in the sense that it searches
for a complementary pair of change-points; but the statistic 
$S_j - S_i - (j-i)S_m/m$ is not standardized to obtain a  
statistic having a marginal 
distribution with unit variance.  Its false 
positive probability is poorly controlled when the procedure is iterated
with a fixed threshold, and it  
also has  
little power to detect two change-points that move
in opposite directions. 
These up-down (or down-up) pairs occur frequently in the data
motivating our studies, although the difficulties they pose may not be 
regarded a 
serious liability in other scientific contexts, where such changes 
may be regarded as unlikely. 

\subsection{Simulations}
The first example in Table \ref{tableexamples} is
a modified version of a suggestion of (\cite{OlVe04}), which
those authors said was typical of the copy number data that
motivated their study.   There are three hundred observations and 
four change-points at 138, 199, 208, and 232, with
mean values 0.0, 0.75, 2.5, 0.25, and 1.5 in the five gaps
between change-points.
According to the local power
approximation of the preceding section, LLR has
local power 0.77, 0.73, 0.91, and 0.85, respectively, to detect these
change-points, so
its expected number of
change-points detected is 3.3.  As a reflection of its lower 
threshold SLLR has an expected number of 3.45 detections, 
although,
as we remarked above and see empirically in Table 6 below, it also has a larger rate of 
false positives.  Simulations not reported here indicate  
that these approximate expected values are quite accurate. 

The second example in the table has changes of the same magnitude in the same locations, 
but with all changes in a positive direction.  The results are similar
in spite of the fact that both CBS and Multi are not designed with 
this case in mind.  The third case is qualitatively similar to the first one,
but it contains one very short up-down pair of change-points. 
In this case the expected number of change-points detected by  LLR  is predicted  by our local power
approximation to be 3.4. 

Failure to detect a change-point is marked in the
table  by a zero(0),
and false positives by an asterisk(*).

\medskip
In Table \ref{tableexamples} our implementation of both LLR and SaRa  
was to choose the values 
$j$ by minimizing the associated length of the background $k-i$ 
from among those values of $|Z_{i,j,k}|$ exceeding the threshold.
If necessary, we enforced the condition mentioned above that the
backgrounds not overlap. 
The other possibility mentioned above, to choose the largest value
of the statistic, but then enforce the no overlap condition for the
background values, frequently leads to more
computation but only occasionally leads to a substantial difference in the segmentation.  
For determining joint confidence regions for the change-points, 
we must condition on estimates of the sizes of 
the changes in mean, and
for that purpose the locally largest $|Z|$-values may be more useful, since it is 
usually based on a longer background, hence a larger ``effective sample size.''

\begin{table}[htp]
\caption{\label{tableexamples} Examples of segmentations: $m = 300, \; b_{\rm LLR} = 4.68, \;  b_{\rm SLLR} = 4.21, \;
b_{\rm SaRa} = 4.27, \;
b_{\rm CBS} = 4.23,$ and $b_{\rm Multi} = 1.51.$  
The initial mean value is 0.  Locations of
change-points and mean values after the change are as indicated.} 
\begin{center}
\begin{tabular}{|c|c|c|c|c|}
\hline
Procedure/Parameters & 138, 0.75  & 199, 2.5 & 208, 0.25  & 232, 1.5  \\ \hline
LLR & 164   & 198 & 206  & 248  \\
SLLR & 134  & 197  & 206  & 248  \\
SaRa & 48*, 140  & 198  & 207  & 249  \\ 
CBS & 149  & 199  & 297  & 249 \\
Multi & 149  & 199 & 208 & 249 \\ \hline
LLR & 127  &198  &211  &230  \\
SLLR &127  &199  &209  &230  \\
SaRa  & 130   & 0  & 211  &230  \\ 
CBS  &135  &199  &212  &231  \\
Multi & 135 &199  &212  &231  \\ \hline
Procedure/Parameters & 138, 0.75 & 199, 2.5 & 208, 4.75 & 232, 6.0 \\ \hline
LLR & 145 & 198 & 207 & 234 \\
SLLR & 134 & 197 & 206 & 232 \\
SaRa  & 134 & 0 & 207 & 234 \\
CBS & 140 & 199 & 208 & 235 \\
Multi & 140 & 199 & 208 & 235 \\ \hline
LLR & 137 & 0 & 207 & 231 \\
SLLR & 136 & 197 & 206 & 235 \\
SaRa  & 137 & 199 & 208 & 235 \\
CBS & 138 & 0 & 207 & 236 \\
Multi & 138 & 0 & 207 & 236 \\ \hline
LLR & 159 & 198 & 207 & 231 \\
SLLR & 129 & 198 & 209 & 231 \\
SaRa & 131 & 198 & 0 & 231 \\
CBS & 130 & 199 & 208 & 232 \\
Multi & 130 & 199 & 208 & 232 \\ \hline
Procedure/Parameters & 100, 3.0 & 103, -0.5 & 120, 1.8 & 200, 2.5 \\ \hline
LLR & 97 & 102 & 119 & 199 \\
SLLR & 97 & 101 & 117 & 199 \\
SaRa & 0 & 0 & 119 & 200 \\
CBS & 98 & 103 & 120 & 200\\
Multi & 0 & 0 & 120 & 200 \\ \hline
LLR & 98 & 102 & 119 & 200 \\
SLLR & 96 & 101 & 118 & 199 \\
SaRa & 0 & 104 & 119 & 200 \\
CBS & 100 & 103 & 120 & 201 \\
Multi & 0 & 0 & 120 & 207 \\ \hline
LLR & 99 & 101 & 122 & 214 \\
SLLR & 98 & 0 & 117 & 214  \\
SaRa & 99 & 0 & 121 & 214  \\
CBS & 100 & 102 & 122 & 140*,215 \\
Multi &100 & 102 & 122 & 140*, 215  \\ \hline
\end{tabular}
\end{center}
\end{table}

Although the table contains only a few examples, several entries reinforce 
our intuition.  The procedure SaRa lacks power to detect
both of two nearby changes by virtue of its requirement to use a symmetric background.
Compared to SaRa, LLR appears to be better at detecting nearby 
change-points at some loss of power to detect relatively isolated 
change-points.  The procedure Multi fails to detect a short interval
that CBS detects---not surprising since its justification involved
an increase in power to detect longer intervals paid for by  a decrease
in power to detect very short intervals.  

In Table \ref{tableexamples}, we see only a few 
false positive errors. 
For CBS and Multi there is
a false positive error that 
occurred when searching an interval where there is only one true change-point
to be detected.  The statistics detect two, and the incorrect detection
is not eliminated by the 5\% rule described above.  
Other simulations suggest that
this is the most commonly occurring false positive error of those statistics.

As mentioned in Section 1, in studying CNV various authors starting 
with \cite{OlVe04} have found
technical artifacts in the form of local trends that tend to 
disrupt the idealized model
of a step function mean value.  The local trends appear to be 
affected primarily by 
CG content, which oscillates in a roughly sinusoidal fashion.  To test 
robustness against 
these perturbations \cite{OlVe04}
suggest adding a low frequency sinusoid, which produces some degradation of
performance.  In Table \ref{localtrends}, we report a very small 
simulation comparing LLR to CBS
in the presence of a sinusoidal perturbation of the mean values.   
In the first three rows, the  amplitude and frequency are larger than
those suggested by \cite{OlVe04}.  In the second three rows, the
amplitude is still larger, the frequency is relatively small, and a random 
phase has been included.  These and other simulations, not shown here, 
suggest that modest local trends lead to slight increases in  
the false positive rate of CBS (and Multi), but not  LLR,  and to slight
decreases in the power of detection of all methods.
The local trends in the last three rows have a large amplitude and
small frequency.  Without these local trends, the change-points 
would be easy to detect,
and the up-down pairs are ideal for CBS. The local power approximation of
Section 3.2 indicates that in the absence of the local trends local power to detect each of the four
change-points averages about 0.95.  Indeed, each change-point is detected,
but there is a striking increase in false positives for CBS (and Multi). 

\begin{table}[h] 
\caption{\label{localtrends} 
Examples With Sinusoidal Local Trends: $m = 200, \; 
b_{\rm LLR} = 4.54, \; 
b_{\rm CBS} = 4.13.$
Format as in Table \ref{tableexamples}, but to simulate 
local trends $0.2 \sin(0.1k)$
is added to the $k$th mean value in the first three rows. For the second three rows
the $k$th mean value is $0.4 \sin (0.05 k + U)$, where $U$ is a uniformly distributed random phase.
For the third three rows, the local trend is $0.7\sin(0.03k + U)$.}
\begin{center}
\begin{tabular}{|c|c|c|c|c|}
\hline
Procedure/Parameters & 60, 3.0  & 63, -0.2 & 83, 1.1  & 153, 2.0  \\ \hline
LLR & 59   & 63 & 83  & 152  \\
CBS & 60 & 63  & 83  & 152 \\ \hline
LLR & 60  & 63  & 78  &155  \\
CBS  &60  &63  &83, 141*  &155  \\ \hline
LLR & 60 & 63 & 99 & 0 \\
CBS  & 60 & 63 & 99 & 0 \\ \hline
Procedure/Parameters & 50, 1.5 & 65, -0.1 & 125, 1.2 & 145, 2.6\\ \hline
LLR  & 42 & 66 & 127 & 144 \\
CBS & 42 & 65,73* & 126 & 144 \\ \hline
LLR & 43 & 0 & 130 & 0 \\
CBS & 25*,43 & 0  & 0 & 143 \\ \hline
LLR & 50 & 65 & 122 & 145 \\ 
CBS & 50 & 65 & 0 & 144 \\ \hline
Procedure/Parameters & 50, 2.0 & 60, 0.0 & 135, 3.0 & 140,0.0 \\ \hline
LLR  & 50 & 60 & 135 & 140 \\
CBS & 50 & 60 & 121*, 135 & 140, 181* \\ \hline
LLR  & 50 & 60 & 135 & 140 \\
CBS & 47 & 60, 82* & 135 & 140, 184* \\ \hline
LLR  & 50 & 60 & 135 & 140 \\
CBS & 50 & 60 & 135 & 140 \\ \hline
\end{tabular}
\end{center}
\end{table}

Table \ref{tableMC} provides the outcomes of 1000 simulations 
for detecting $M$ change-points
randomly located from 0 to 500. The sizes of the changes are
normally distributed with mean value $2.5 \xi$, where the values 
$\xi$ are independently  $\pm 1$ with probability 1/2 and 
variance $0.5.$  The first  
method uses the LLR statistic with segmentation based on the 
smallest value of $k-i$ for which the statistic exceeds the
0.05 level threshold $b_{\rm LLR} = 4.83$;  the second is a faster
LLR procedure  introduced in the following paragraph; the third is 
the sequential version described above, with the threshold 4.33;
the fourth is the Wild Binary Segmentation (WBS) procedure of \cite{Fr14} with
5,000 random segments and
threshold $b_{\rm WBS}$ = 4.565. 
(This threshold is close to, but slightly different
from the value $1.3 [2\log(m)]^{1/2} = 4.58$ recommended by
Fryzlewicz, which was presumably determined by numerical
experimentation.  Our threshold would be the 0.15 significance threshold 
according
to the approximation \eq{2.1}.) 
The fifth is SaRa with the (simulated) threshold 4.42.
The last two procedures are CBS and Multi as discussed above.

Since LLR requires order $m^3$ computations, it can 
be slow for large values of $m$.
A possible speed-up is based on the observation 
that for the large backgrounds required to detect 
relatively small changes, determining the
exact background does not seem to be important.  Suppose that in
considering a fixed value of $j$,
to determine an appropriate $k$, we choose $k = j+1$,
then choose a new value of $k$ recursively as the old value plus
$\max(1, [(k-j)/10])$, where $[x]$ denotes the largest integer less than or equal to $x$.
Thus, for $k-j < 20$, we choose every integer, then every second integer for $k-j < 30$,
etc.  The computational complexity of this procedure is 
of order $m (\log(m))^2 \ell^2$.  In
Table \ref{tableMC} this procedure with $\ell = 10$ is denoted LLR-F.
Other speed-ups of a similar nature are possible.


\begin{table}[h] 
\caption{\label{tableMC} Random change-points, $m = 500$, $b_{\rm LLR} = 4.83$, $b_{\rm LLR-F}=4.83$, $b_{\rm SLLR} = 4.33$,
$b_{\rm WBS} = 4.565$, $b_{\rm SaRa}=4.42$, $b_{\rm CBS} = 4.36,$ $b_{\rm Multi} = 1.57.$  
The three rows in each entry are the number of times that the number of change-points is correctly  detected, under detected and over detected, respectively, in 1,000 repetitions. 
The accompanying  numbers in parentheses are
the number of change-points missed (false negative errors) and the number of
over detections (false positive errors). 
E(asy) denotes the number of
repetitions where all methods detected the correct number of change-points;
I(mpossible) gives the number of repetitions where no method detected 
the correct number of
change-points.}

\begin{center}
\begin{tabular}{|c|c|c|c|c|c|c|c|c|}
\hline
$M$ & LLR & LLR-F & SLLR & WBS & SaRa & CBS & Multi &E/I \\ \hline
0 &954 & 959 & 957 & 933 & 960 & 952 & 960 & 881/5 \\
& 0(0) & 0(0) & 0(0) & 0(0) & 0(0) & 0(0) & 0(0) & \\
& 46(53) & 41(47) & 43(77) & 67(86) & 40(41) & 48(99) & 40(73) & \\ \hline
3 & 843 & 842 & 846 & 819 & 806 & 807 & 826 & 662/72  \\ 
& 121(126)  & 123(129) & 83(87) & 112(116) & 152(161) & 93(98) & 79(83) &  \\ 
& 36(37) & 35(36) & 71(95) & 69(80) & 42(45) & 100(157) & 95(127) &  \\  \hline
5 & 683 & 674 & 714 & 691 & 598 & 665 & 660 & 440/180 \\
 & 299(355)  & 309(366) & 227(270) & 253(302) & 383(470) & 217(260) & 215(264) & \\
  & 18(20) & 17(19) & 59(76) & 56(60) & 19(19) & 118(155) & 125(157) & \\ \hline
8 & 362 & 355 & 487 & 392 & 274 & 425 & 427 & 170/359 \\
 & 626(935) & 635(949)  & 471(678) & 565(849) & 719(1171) & 422(619) & 443(655) & \\
  &12(14) & 10(12) & 42(48) & 43(49) & 7(7) & 153(205) & 130(164) & \\ \hline
\end{tabular}
\end{center}
\end{table}

The procedures LLR and SaRa are the only ones that in theory have 
control over the 
false positive rate; and the simulations  show that  others  
can make twice as many false positives errors.
As expected SaRa
has problems with detection of near-by change-points.  
Both these procedures pay for their strict control over the false positive rate 
with less power than the others, 
especially when there is a relatively large number of
change-points.  
CBS and Multi do not lose as much power as the number of change-points increases, 
but they suffer an increasing number of false positive errors.
SLLR and WBS do not have the false positive control of LLR and SaRa, but they 
also do not suffer as great a loss of power with an increasing
number of change-points.  Between SLLR and WBS, SLLR seems to perform
slightly better. 

Although the simple counts in Table \ref{tableMC}  without an indication of
accuracy of the detections are not definitive, as we see in Tables
\ref{tableexamples} and \ref{localtrends}, in most cases
accuracy of estimation of the change-points is less an issue than the errors of over or under detection.

\subsection{Array CGH data}

In this section we present examples involving changes in
copy number from array CGH data.  

We first consider the test cases GBM29 and GBM31 used by 
\cite{La05}, to compare different methods of segmentation.

For GBM29, the total length of the sequence is 193.
The estimated standard deviation is 0.76.
The theoretical 0.05 thresholds for 
LLR and SLLR are 4.53 and 4.07, 
while that for CBS is 4.12 and for Multi is 1.45.
Change-points are detected
at 
\be{
81, 85, 89, 96, 123, 133
}
by all  methods.

For GBM31, the length of the sequence is 797. 
The estimated standard 
deviation is 0.38.  All methods, except 
the multiscale statistic detect the same set of change-points, at
\be{
317, 318, 538, 727, 728.
} 

The third change-point is a relatively 
small change apparently indicating a long region of
loss of copy number; the 
first two and
last two change-points are large spikes.   
Only one of the two is
detected by the multiscale statistic,  which is designed to
favor detection of longer intervals.  



We have also tested our methods on the BT474 cell line data from
\cite{Snijders03}.  See \cite{Pollack99, Pollack02} for a
different experimental technique involving BT474 and a discussion
of the implications for breast cancer.  
This cell line has also been used by,
e.g., \cite{Zhao04}, who based their experimental
technique on SNPs rather than
array CGH.

For a scan of the entire genome, which involves slightly
more than 2000 observations,
we detect 63 change-points with LLR at a 0.05 
genome wide significance 
threshold
of $b = 5.2$; and we detect 67 using the pseudo-sequential procedure
with a threshold of 4.7.  
However, the data are organized by chromosomal location, and it turns out
that the estimated standard deviation varies considerably from
chromosome to chromosome.  Although the cited literature typically 
involves scans of the entire genome, we find a scan of each chromosome
using the estimated standard deviation of that chromosome more reasonable.

We continue to use genome wide thresholds, which are 
4.68 for CBS and 1.67 for Multi; but we now use standard
deviations specific to each chromosome.
Particularly interesting are chromosome 17, where an increase
in copy number 
appears to have implications for the severity of breast cancer, and chromosome
20, which appears to contain a second 
increase in copy number embedded in a modest
increase in copy number. 
For 
chromosome 17, there are $m = 87$ observations, with an estimated standard deviation 
of 0.51.  According to
LLR, SLLR, CBS and Multi, there is
an increase in copy number at the 35th observation (17q11.2-12), 
with a change back to baseline just two observations later. There is
a second increase at the 50th observation (17q21.3) and a 
return to the baseline at the 66th (17q23). Chromosome 20 contains
$m = 85$ observations, and the standard deviation is 0.59.  
LLR, SLLR, CBS, and Multi  again agree and detect a 
{\sl decrease}
in copy number from the 38th (20q11.2) to the 52nd observation,
followed by
an increase from the 53rd (20q13) to the 68th (20q13.1).  From
the 69th observation there
is an even larger increase until the 82nd (20q13.3), 
then a return to roughly the baseline value for the last three observations.

Also interesting are chromosomes 4, 5 and
11, all of which have several changes, and some of the changes are followed by
a second change after only a few observations.  On
chromosome 4 there are 162 observations and an
estimated standard deviation of 0.19. At the 0.05 
global significance level LLR and SLLR detected changes 
at 7, 8, 59, 61, 141, 143, and 155.  CBS and Multi detected the
same changes with the exception of 143, which both missed.
On chromosome 5 there were 99 observations and an estimated
standard deviation of 0.16. Changes were detected 
by all four methods at 
25, 45, 51,  and 65.  CBS and Multi also detected paired changes
at 87 and 91. The first of these was missed by LLR, and both 
were missed by SLLR.
On chromosome 11 there are 181 observations and an estimated 
standard deviation of 0.34.  Changes were detected
by all four methods
at 91, 124, 139, 144, 162, 165.  In this case SLLR also detected 
changes at 6 and 163.  Looking at a plot of the cumulative sum of the data
and the proximity of the statistic to the detection
threshold suggests that the change at
163 is a false positive.  The putative change at 6 is 
also borderline, but looks real in the cumulative sum plot.

To illustrate our confidence region calculations, we consider Chromosome
3, where there are 85 observations and change-points are detected
at 19, 39, and 44.  The estimated size of the change at 
44 is $\hat{\delta} = 2.25$,
while the changes at 19 and at 39 are estimated to be  substantially larger.
For simplicity we (conservatively)
use the single estimated difference, $\hat{\delta} = 2.25$, 
so from the theory developed above, the critical constant for a 
95\% joint conditional confidence region for the three change-points
is 4.63.  Using this threshold, a joint confidence region consists of
the exact point estimates 19 and 39, and the union of 43, 44, and 45.
For Chromosome 15 change-points are detected at observations 43 and 57,
where the smaller change is estimated to be about 2.3 and the other 
only slightly larger.  For the approximate threshold of $ b = 3.6$,
we found a 95\% joint confidence region to consist of the four pairs
42 or 43 and 56 or 57. 
For Chromosome 20, where we detected change-points at 38, 52, 68, and 82,
the smallest value of  $\hat{\delta}$ is 2.1 at 68.  Using this 
single estimator, we find that the critical 
constant for a 95 \% joint confidence region for four change-points is $b = 5.9.$
The union of the values that 
in various 4-tuples form the joint confidence region
are 38, 39, 51, 52, 66, 67, 68, and 82.  

\smallskip\noindent{\bf Remark.}  In studying copy number variation it is 
customary to plot the locus by locus measurements, which should be about 
equal to zero when
the copy number is two, with positive values indicative of amplifications
and negative values indicative of deletions.  There may be advantages 
to plotting the consecutive partial sums also and looking for
a change in slope to indicate an increase or decrease in copy number.
This plot is substantially smoother, and changes in slope that are 
candidates for change-points in copy number are often easier to see
than in a plot of the raw data.  The
disadvantage is that it is sometimes difficult to infer the
regions of normal copy number, which are regions where the 
slope is should be zero although it seems that it is always
different from zero.

\subsection{Simulations for Confidence Intervals}

In order to illustrate the size of the joint confidence regions introduced in
Section 3, we consider
in Table \ref{table10} some parameter settings related to Table 3.
The upper part of the table, like Table \ref{table9},  gives the estimated
coverage probability based on 10000 simulations for examples where the
threshold $b$ has been
selected so our theoretical approximation gives the
probability 0.05.  The lower part of the table gives the probability
from 1000 simulations that the indicated values of $t_1, \; t_2$  are
{\sl not} contained in the confidence region.  We have chosen values
of $t_i$ for which this probability is about 0.5, so one can regard the
difference between $t_i$ and $\tau_i$  as a rough measure of the size of
the confidence region
when all other parameters are set to their correct values.  Recall that
$\delta_i = \mu_i - \mu_{i-1}$ denotes the size of the change at $\tau_i$.

The rows beginning with 0.65 are particularly interesting, since they
show that the relatively small change at $\tau_1 = 138$
compared with very large change at
$\tau_2 = 225$
leads to substantially more uncertainty in the value of
$\tau_1$ compared to the value of $\tau_2.$
\begin{table}[h]
\caption{Likelihood ratio based joint confidence
intervals.  $\hat{p}$ is the simulated probability that the parameters $t_1$ and
$t_2$ are rejected when the true parameter values are $\tau_1$ and $\tau_2$.
Nominal confidence level is 0.05. Simulations are based on 10000 (1000)
repetitions in the first four (last 12) rows. }
\label{table10}
\begin{center}
\begin{tabular}{c|c|c|c|c|c}
\hline
$\delta_1$ & $\delta_2$ &  $a$ & $\tau_1,\; \tau_2$ & $t_1,\;t_2$&$\hat{p}$ (Monte Carlo) \\
\hline
2.13& 1.33 & 6.4 &9, 33 & 9, 33& 0.049 \\
2.5 & 4.0 & 5.35 & 87, 104 &87, 104& 0.051\\
0.65 & 2.5 & 6.65 & 138, 225 & 138, 225 & 0.047 \\
1.73 & 2.13 & 6.23 & 57, 66 & 57, 66 & 0.049\\ \hline
2.13 & 1.33 & 6.4 & 9, 33& 7, 33& 0.59 \\
2.13 & 1.33 & 6.4 & 9, 33 & 11, 33 & 0.58\\
2.13 & 1.33 & 6.4 & 9, 33 & 9, 29 & 0.47 \\
2.13 & 1.33 & 6.4 & 9, 33 & 9, 37 & 0.44 \\
0.65 & 2.5 & 6.65 & 138, 225 & 138, 227 & 0.75\\
0.65 & 2.5 & 6.65 & 138, 225 & 138, 223 & 0.73 \\
0.65 & 2.5 & 6.65 & 138, 225 & 120, 225 & 0.49 \\
0.65 & 2.5 & 6.65 & 138, 225 & 156, 225 & 0.46 \\
2.5 & 4.0 & 5.35 & 87, 104 & 87, 102& 0.43 \\
2.5 & 4.0 & 5.35 & 87, 104 & 87, 106 & 0.44 \\
2.5 & 4.0 & 5.35 & 87, 104 & 86, 104 & 0.89 \\
2.5 & 4.0 & 5.35 & 87, 104 & 88, 104 & 0.89 \\ \hline
\end{tabular}
\end{center}
\end{table}

\subsection{Comparison with other confidence intervals}

\cite{Mu14} suggested a different  method to construct a
confidence region jointly for the change-points and the mean values
of the observations in the segments connecting those change-points.
For each candidate set of change-points $\tau$ and mean values $\mu$,
they suggest an application of their multiscale statistic
\ben{\label{13}
\max \big( \frac{|S_j-S_i-(j-i)\mu|}{(j-i)^{1/2}}-[2 \log(3m/(j-i))]^{1/2}  \big)
}
where the maximum is taken over all $i<j$ within one of the segments
of $(0,\tau_1], \dots, (\tau_M, m]$, and $\mu$ is the hypothesized
mean value in the segment.
This is in effect a test of the hypothesis that there
are no change-points in the hypothesized
segments $(0,\tau_1], \dots, (\tau_M, m]$ and the mean values are
as hypothesized.
\cite{Wo86} discusses a similar idea under the assumption that there
is a single change-point, and one is interested only in a confidence region
for the change-point, not a joint confidence region for change-points and
means.  (Note that our approximation \eq{201} allows us to condition on the
sum of the observations in the interval under investigation and hence
use these ideas to obtain joint confidence regions for the change-points alone.)

It is difficult to make a comparison of the two methods.
In Table \ref{table13} we compare our confidence region defined by \eq{2} with that using \eq{13}
in a small number of examples.
We set $m = 200,$  $\tau_1 = 50, \; \tau_2 = 100$ and consider
values of the $\delta_i$ that
are large enough that most of the time we will detect  two
change-points.  The problem
becomes one of locating them and estimating the mean values.
For our confidence regions, we choose the thresholds
$b_1=7.2$ so
that the probability in \eq{6} equals 0.05.  This threshold was
confirmed by simulation. Moreover, for the statistic \eq{13}, we chose
the threshold $b_2=1.44$ for which a
20000 repetition simulation experiment gave the
probability  0.05.  This threshold is slightly larger than the
theoretical approximation 1.41.


Since a direct comparison of these regions in terms of size is conceptually
complicated and
technically demanding, we use the relation of confidence regions
to  hypothesis testing to compare them in terms of power.  Under specific
hypothetical, but incorrect, values of the change-points and mean values
the power of the test of the true values represents the probability that
the hypothetical values do not lie in the confidence region.  Hence the
procedure with larger power is preferred.
From Table \ref{table13}, it seems clear that for the
parameter settings analysed, the
likelihood ratio procedure is preferable.


\begin{table}[h]
\caption{Power to Detect Departure from True Parameter Values:
$\tau = (50,100)$ and $\mu$ as given; $t$ and $\xi$ are hypothesized
values of $\tau$ and $\mu$. The subscript 1 indicates the likelihood ratio
procedure, while 2 indicates the procedure based on \eq{13}.
Simulations are based on 10000 repetitions.}
\label{table13}
\begin{center}
\begin{tabular}{c|c|c|c|c}
\hline
$\mu$ & $\xi$ & $t$ &  $\widehat{\rm Power}_1$ (Monte Carlo) & $\widehat{\rm Power}_2$ (Monte Carlo)\\ \hline
0.0,1.0,0.0 & 0.0,1.0,0.0 & 55,95 & 0.64 & 0.08 \\
0.0,1.0,0.0 & 0.1,0.9,-0.2 & 55,95,& 0.87 & 0.40 \\
0.0,1.0,0.0 & 0.1,0.9,-0.2 & 40,100 & 0.75 & 0.32 \\
0.0,1.2,2.0 & 0.0,1.2,2.0 & 47,105 & 0.47 & 0.044 \\
0.0,1.2,2.0 & 0.0,1.5,1.9 & 47,105 & 0.75 & 0.32 \\
0.0,1.5,0.75& 0.1,1.4,0.9 & 40,97 & 0.96 & 0.69 \\
0.0,1.5,0.75 & 0.0,1.5,0.75 & 44,98 & 0.81 & 0.32 \\
0.0,1.2,-0.1 & 0.1,1.1,0.1 & 48,103 & 0.68 & 0.22 \\
0.0,1.1,0.1 & -0.2,1.0,0.0 & 52,115 & 0.91 & 0.44 \\
0.0,1.0,2.0 & -0.1,1.1,2.1 & 45,110& 0.87 & 0.24 \\ \hline
\end{tabular}
\end{center}
\end{table}

\section{Exponential Families}
A natural generalization of the methods of this paper involve data from 
exponential families, where there usually is the option to pursue analogous 
methods or to use a normal approximation.  We first develop the analogous theory
and discuss the second possibility below.  

Assume $X_1,\dots, X_m$ are independent and from a one-parameter exponential family of distributions $\{F_\theta: \theta\in \Theta\}$ where
\be{
\frac{d F_{\theta}}{d u} (x)=\exp(\theta x -\psi(\theta)), \quad x\in \mathbb{R},\  \theta\in \Theta,
}
$u$ is a $\sigma$-finite measure on the real line and $\Theta$ is an open interval.
For $0\leq i<j<k\leq m$, the likelihood ratio statistic to test whether $j$ is a change-point in the local background $(i,k)$ is
\bes{
\ell_{i,j,k}=&(j-i)\sup_{\theta_1\in \Theta}(\theta_1 \frac{S_j-S_i}{j-i}-\psi(\theta_1))
+(k-j)\sup_{\theta_2\in \Theta}(\theta_2 \frac{S_k-S_j}{k-j}-\psi(\theta_2))\\
&-(k-i)\sup_{\theta\in \Theta}(\theta \frac{S_k-S_i}{k-i}-\psi(\theta)).
}
In the following, we use $\p_\theta$ ($\E_\theta$ resp.) to denote the probability (expectation resp.) calculated when $X_i \sim F_\theta,\ \forall \ i$.
Following the proof of \eq{2.1}, we suggest the following approximation to the $p$-value of $\max_{i,j,k}\ell_{i,j,k}$:
\besn{\label{14}
&\p_\theta(\max_{i<j<k \atop m_0\leq j-i, k-j\leq m_1}\ell_{i,j,k} \geq \frac{b^2}{2})\\
\sim & \varphi(b)
\sum_{m_0\leq n_1, n_2\leq m_1:\atop n_1+n_2\leq m}
(m-n_1-n_2)\sum_{\theta_1,\theta_2}
\frac{a(\theta_1,\theta)a(\theta_1,\theta_2)a(\theta, \theta_2)}{[n_1(\theta_1-\theta)^2\psi''(\theta_1)+n_2(\theta_2-\theta)^2\psi''(\theta_2)]^{1/2}}
}
where the second summation is over two pairs of $\theta_1<\theta_2$, which are assumed to exist (see the remark below), solving
\ben{\label{15}
\begin{cases}
\psi'(\theta_1)n_1+\psi'(\theta_2)n_2=\psi'(\theta)(n_1+n_2),\\
n_1[\theta_1 \psi'(\theta_1)-\psi(\theta_1)]+n_2[\theta_2 \psi'(\theta_2)-\psi(\theta_2)]-(n_1+n_2)[\theta \psi'(\theta)-\psi(\theta)]=b^2/2,
\end{cases}
}
and for $\theta_1<\theta_2$,
\be{
a(\theta_1,\theta_2)=\exp(-\sum_{1}^{\infty} n^{-1} \E_{\theta_2} e^{-[(\theta_2-\theta_1)S_n-n(\psi(\theta_2)-\psi(\theta_1))]^+}).
}
We use Theorem 8.51 of Siegmund (1985) and Theorem A of \cite{TuSi99} to compute $a(\theta_1,\theta_2)$ numerically for nonarithmetic and arithmetic random variables respectively.

\noindent{\bf Remark.} For those $n_1$ and $n_2$ such that the solutions to \eq{15} do not 
exist, we first find the smallest $\theta'>\theta$ such that the solutions to \eq{15} with 
$\theta$ replaced by $\theta'$ exist. We denote the solutions by $\theta_1'$ and $\theta_2'$. 
Then the proposed approximation is the RHS\eq{14} with $\theta, \theta_1, \theta_2$ replaced 
by $\theta', \theta_1', \theta_2'$ respectively, and multiplied by $\p_\theta(S_{n_1+n_2}/(n_1+n_2)\geq \psi'(\theta'))$.

\subsection{Simulations}

We first consider the exponential distribution with rate $\lambda$. Observing that in \eq{14}, both the probability and its approximation do not depend on $\lambda$, we choose $\lambda=1$ without loss of generality. We fix $m_0=1$. In Table \ref{table21}, with different values of $m, m_1$ and $b$,  $p$ denotes the RHS\eq{14}
and $\hat{p}$ denotes the simulated $p$-value with 2000 repetitions.
We see from Table \ref{table21} that our approximation to the $p$-values 
are reasonably accurate, especially when $m$ and $m_1$ are large.
A normal approximation would also be quite reasonable,
especially for larger $m_1$ and $m$.  For example, for the last line of Table \ref{table21}
our normal approximation gives the probability 0.053.

\begin{table}[h]
\caption{Exponential distribution with rate $\lambda$.}
\label{table21}
\begin{center}
\begin{tabular}{c|c|c|c|c|c}
\hline
$\lambda$ & $m$ & $m_1$ & $b$ & $p_{\rm Approx}$ &   $\hat{p}$ (Monte Carlo) \\ 
\hline
1 & 500 & 50 & 4.72 & 0.049 & 0.061\\
1 & 500 & 100 & 4.78 & 0.048 & 0.053  \\
1 & 1000 & 100 & 4.95 & 0.047 & 0.048\\ \hline
\end{tabular}
\end{center}
\end{table}

Next, we consider the inverse Gaussian distribution with fixed shape parameter $\lambda=10$. We fix $m_0=1$. With different values of the mean $\mu$, $m, m_1$ and $b$, $p$ denotes the RHS\eq{14} and $\hat{p}$ denotes the simulated $p$-value with 2000 repetitions. We can see from Table \ref{table12} that both the theoretical and simulated $p$-values are reasonably robust against the mean $\mu$.

\begin{table}[h]
\caption{Inverse Gaussian distribution with shape parameter $\lambda=10$.}
\label{table12}
\begin{center}
\begin{tabular}{c|c|c|c|c|c}
\hline
$\mu$ & $m$ & $m_1$ & $b$ & $p_{\rm Approx}$ &   $\hat{p}$ (Monte Carlo) \\ 
\hline
1 & 300 & 30 & 4.5 & 0.059 & 0.053\\
5 & 300 & 30 &4.5 & 0.041 & 0.050\\ 
1 & 500 & 100 & 4.78 & 0.049 & 0.037\\
5 & 500 & 100 & 4.78 & 0.035 & 0.031\\
1 & 1000 & 100 & 4.95 & 0.049 & 0.050\\
5 & 1000 & 100 & 4.95 & 0.036 & 0.034\\
\hline
\end{tabular}
\end{center}
\end{table}

Since the computation of appropriate thresholds for non-normal exponential families
is somewhat complicated, one may also consider the use of normal approximations, which 
in these cases would work quite well.   Following are two examples where a Gaussian approximation
to the signed square root of the likelihood ratio statistic seems to perform admirably.

For the detection 
of CG rich regions in  genomic studies, as mentioned in the introduction, 
the sequences are very long and 
the exact boundary between regions has little biological significance. Hence 
one often forms groups of consecutive Bernoulli variables. Following 
\cite{EGJ10}, we have used groups of 33 consecutive 
Bernoulli variables. Since the values of the Bernoulli parameters $p$ are 
usually neither extremely small nor extremely large, possibilities that  might indicate 
a Poisson approximation, we have tentatively assumed that we can use 
the theory developed above for the normal distribution. Since the Bernoulli 
variances must be estimated locally in each homogeneous region, it turns out 
that  the skewness of the binomial distribution when $p$ is not in the 
immediate neighborhood of 1/2 can make an approximation of the 
distribution of the scaled value of $[S_j - S_i - (j-i)(S_k - S_i)/(k-i)]$
by a normal distribution unsatisfactory, unless the size of the groups is relatively large. 
Consequently we have used the signed square roots of the log likelihood 
ratio statistics, which behave very much like a Gaussian process. 
It turns out that 
simulations of this process indicate that the
approximation is quite  satisfactory and offer no new insights, so
we omit the details. 

The copy number data discussed in this paper was all obtained by comparative genomic
hybridization.  To achieve greater resolution, many present day studies use sequence
data (e.g., Zhang {\sl et al.} 2016), which often utilize models built from Poisson processes.
The simplest of these is concerned with detection of a change from a background rate
for a Poisson process.  Since the background rate varies with genomic position due
to variation in sequencing depth,
local detection
procedures along the lines of LLR may be useful.  
Like the binomial distribution, to
detect changes in the rate of a Poisson process, simulations
support an approximation based on a normal approximation to the signed square root of
the (generalized) log likelihood ratio statistic.  For 500 observations, $b = 4.83$,  
and the mean of the
Poisson distribution equal to 10, 400 simulations gave the significance value 0.0475,
when our normal approximation gives the value 0.05.   Calculation of Kullback-Leibler
information suggests that for detecting 
changes from 10 to 20 and back to 10 in well separated intervals, interval lengths of 6 and 7 
are borderline detectable.  Several simulations of this case involving two pairs
of change-points lead to successful
detections of all four change-points,
while the differences between the 
estimates of the change-points and the true values totaled 1-3 observations.

\subsection{Changes in a Normal Mean and Variance}
An interesting, but considerably more complex example, is to allow for simultaneous changes
to both the mean and variance (or mean vector and
covariance matrix) of a sequence of independent, normally distributed observations.
Although the formulation we have adopted, which assumes a constant value of the variance is much more common, 
and the copy number data considered above shows little evidence of heteroscedasticity within chromosomes,
the recent paper (\cite{DKK16}), where the possibility of simultaneous changes in the mean and
variance is considered, motivates the following brief discussion. 

For $0 \leq i < j \leq m$ let
$\ell_{i,j} = -.5(j-i)\log(\sigma^2) - .5 \sum_{i+1}^j (X_k - \mu)^2/\sigma^2$
denote the log likelihood of $X_{i+1}, \cdots, X_j$, and let 
$\hat{\ell}_{i,j} = -.5(j-i)\log(\hat{\sigma}^2_{i,j}) - .5 
\sum_{i+1}^j (X_k - \bar{X}_{i,j})^2/\hat{\sigma}_{i,j}^2$ denote the
log likelihood with parameters replaced by estimators.  When the 
estimators are the maximum likelihood estimators, the
generalized likelihood ratio statistic (which reduces to 
one-half the square of (\ref{LR}) 
in the case of known $\sigma^2 = 1$) is
$\hat{\ell}_{i,j} + \hat{\ell}_{j,k} - \hat{\ell}_{i,k}$, maximized 
over $i < j < k$.  Necessarily we must take the minimum values of
$j-i$ and $k-j$ at least equal to $m_0 = 2.$ If one is interested in detecting changes 
occurring as close together as those studied above, 
this maximum likelihood ratio statistic is very unstable when
there are no changes and
$j-i$ or $k-j$ is small, since the maximum likelihood estimator of 
$\sigma^2$ can with substantial probability assume very small 
values.   The consequence is that
a suitable threshold to control the rate of false positives must be so
large that the statistic has very poor power to detect changes, and this
problem persists even when $m_0$ is substantially larger than 2.    

A device to ameliorate this problem that maintains the invariance of
the likelihood ratio  statistic under scale and
location changes is to subtract a small constant $c/2$ from the sample 
size in the denominators of the estimators $\hat{\sigma}_{i,j}^2$ 
and $\hat{\sigma}_{j,k}^2$, and subtract $c$ from the denominator of
$\hat{\sigma}_{i,k}^2$.  Then with these
new estimators (denoted by a tilde) use the statistic
$ -(j-i-c/2) \log(\tilde{\sigma}^2_{i,j}) - (k-j -c/2) 
\log(\tilde{\sigma}^2_{j,k})
+ (k-i-c) \log(\tilde{\sigma}^2_{i,k}).$  
In simulations we have found that with $m_0 = 2$ and $c \approx 2.7$, this statistic
has a false positive rate approximately the same as a two dimensional
version of (\ref{LR}), for which the significance level and
power approximations of this paper are easily adapted. 
A similar result holds for the corresponding CBS statistic.
If the variance changes by
a factor of $1+\Delta$, the difference in mean values, scaled to
unit standard deviation, is $\delta$, and $\pi$ denotes the fraction of
observations at unit variance before a change-point,
rough law of large numbers arguments indicate a noncentrality parameter
in large samples proportional to
\[\pi(1-\pi)\log(1 + \pi(1-\pi)\delta^2+(1-\pi)\Delta) - (1-\pi)\log(1+\Delta)
\] 
for the two dimensional statistic. 

If in fact there is no change in the variance the marginal power of 
the two dimensional statistic to detect a change-point
is approximately 0.2 - 0.3 less than the power of (\ref{LR}).  When the variance does change, 
theoretical calculations and  simulations suggest that there is a complex 
tradeoff that depends on the size of the
changes in variance and the relative locations of the various change-points.  Finally, there is 
also the issue that the likelihood ratio statistic that tests for a change in both mean and variance
will not be as robust against excess kurtosis as a statistic that tests only for a change in mean value.

Following are the results of a few simulations that indicate the 
complexity of the problem.  
The statistics considered are the 
two-dimensional statistic suggested in this section, the statistic (\ref{LR}), and a modified
version of (\ref{LR}), designed to compensate for the possibility that (\ref{LR}) has an 
excess of false positives.  Since (\ref{LR}) estimates an average variance, if there is a 
sub-interval where the variance is much larger than that average variance,  the statistic 
(\ref{LR}) will use an inappropriately small variance estimate, which may lead to false 
positives.  The modification of (\ref{LR})  is as follows:  for any $i < k$, when searching for a putative 
change-point in $[i,k]$, standardize the process by the estimated (maximum likelihood)  variance
of the observations $X_i, \ldots, X_k$.  If there is a change-point in the interval, the 
maximum likelihood estimate may be positively biased, but other possibilities appear to be 
too unstable when the interval is short.  Simulations indicate that the thresholds suggested
by Theorem 1 are conservative.  

\begin{table}[h]
\caption{Changes in Mean and Variance: $m = 200$, Threshold for (1.1) and for the
modification suggested above is $b_1 = 4.54$;
threshold for the two dimensional statistic is $b_2 = 4.97.$ Detected change-points are 
as noted for (\ref{LR}), for the modification indicated in the text (denoted by an 
asterisk), and for the two dimensional statistic suggested in this Section, respectively. False positive
errors are denoted by an asterisk.}
\label{table14}
\begin{center}
\begin{tabular}{c|c|c|c|c|c}
\hline
$\tau$ &$\mu$ & $\sigma$ &  (1.1) & $(1.1)^*$ & 2-D  \\
\hline
38,88,108,132 & 1.1,2.7,1.0,2.5 & 1.1,1.8,1.1,1.7 & 0, 88,104,132& 0,86, 0,132 & 0, 88,0,133  \\
38,88,108,132 & 1.1,2.7,1.0,2.5 & 1.1,1.8,1.1,1.7 & 39,86,106,135 & 39,86,106,0 & 39,88,107,0 \\
38,88,108,132 & 1.1,2.7,1.0,2.5 & 1.1,1.8,1.1,1.7 & 0,0,0,132 & 37,0,0,132& 69,82,0,0\\
38,88,108,132 & 1.1,2.7,1.0,2.5 & 1.1,1.8,1.1,1.7 & 0,90,108,134 & 36,90,0,134 & 36,84,0,134\\ \hline
30,80,110,135& 1.5,0.5,2.5,1.0& 1.5,1.0,2.0,1.2& 0,0,110,134 & 0,0,110,0 & 0,0,110,136 \\
30,80,110,135& 1.5,0.5,2.5,1.0& 1.5,1.0,2.0,1.2& 30,74,110,115*,132 & 30,74,110,132 & 30,77,110,133 \\ 
30,80,110,135& 1.5,0.5,2.5,1.0& 1.5,1.0,2.0,1.2& 30, 0,110,134 & 30, 0,110,134 &30, 0,110,134  \\ \hline
48,50,150,154& 4.0,0.0,4.0,0.0 & 2.0,1.0,2.0,1.0 & 48,50,150,154 & 0,0,150,153 & 0,0,150,155 \\
48,50,150,154& 4.0,0.0,4.0,0.0 & 2.0,1.0,2.0,1.0 & 48,50,150,154 & 48,50,150,154 & 48,50,150,155 \\ 
48,50,150,154& 4.0,0.0,4.0,0.0 & 2.0,1.0,2.0,1.0 & 0,0,150,155 & 0,0,150,155  & 0,0,150,154  \\ 
48,50,150,154& 4.0,0.0,4.0,0.0 & 2.0,1.0,2.0,1.0 &48,50,149,154 & 48,50,0,0 & 48,50,0,0 \\ 
48,50,150,154& 5.0,0.0,5.0,0.0 & 2.0,1.0,2.0,1.0 &48,50,150,154 & 48,50,150,154 & 0,0, 150,155 \\ \hline
\hline
\end{tabular}
\end{center}
\end{table}

In Table \ref{table14} the false positive in the sixth row is presumably a reflection
of the fact that in the interval between 110 and 135
the variance of the observations is substantially larger than the ``average variance'' used by 
(\ref{LR}).  Although we did not observe this in a number of other simulations, not reported here, 
this possibility of
an inflated false positive error rate appears to be one of the principal disadvantages of using the 
unmodified (\ref{LR}), which otherwise seems to performs very well.  The last five rows were based on the test 
case suggested by \cite{DKK16} following an earlier suggestion of \cite{La05}, 
but we have reduced the signal to noise ratio to make more difficult what otherwise
would be easy detections.
In those last five rows we see the effect on the two dimensional statistic of the
constant $c \approx 2.7$,
which was introduced to reduce false positive errors in short test intervals, but here 
has an adverse effect on 
the power.  For the 
modified version of (1.1), which often behaves quite similarly to the two dimensional statistic, the
loss of power is presumably due to estimating the variance locally, which leads to large positive biases
in (short) intervals containing change-points. 

Since multiscale methods are designed to favor detection of change-points in longer
over shorter intervals, it
is natural to ask if imposition of a multiscale penalty on the square root of the likelihood ratio 
statistic would work here.  Some numerical experimentation suggests that 
the penalty   
$[4 \log(3m/\min(j-i,k-j))]^{1/2}$ allows one to 
control the false positive rate, and the multiscale statistic performs about as well 
in these examples as the two dimensional statistic defined above. 
 

\section{Discussion}

We have studied local thresholding procedures for segmenting sequences
of independent random variables subject to change-points in the mean.
The local likelihood ratio statistic, LLR, considers for
local subsets of intervals $(i,k)$ the log likelihood ratio statistic
$Z_{i,j,k}$ for detecting 
a change-point at $j$, which is compared to a threshold
designed to control the probability of a false positive error.
The pseudo-sequential procedure SLLR leaves $i$ fixed
at 0 or at the most recently discovered candidate change-point,
then sequentially with respect to $k$ examines
$\max_{i < j < k} Z_{i,j,k}$ until it exceeds a suitable threshold.
The statistic LLR has better false positive control, 
although it requires a relatively large threshold, and hence loses some 
power compared to SLLR, especially when the number of
change-points is large.

Our suggested procedures are compared to 
several other threshold based procedures that attempt to control,
with varying degrees of success, the
false positive error rate:
(i) the Wild Binary Segmentation (WBS) procedure of 
\cite{Fr14}, (ii)
the SaRa procedure of \cite{NZ12}, 
(iii) the CBS procedure of \cite{OlVe04}, and (iv) Multi, a related
iterative threshold based
implementation of the statistic of \cite{Mu14}.
Each of these methods has strengths and weaknesses,  
some obvious, others not so obvious.  
The procedures WBS, SLLR, CBS, and Multi 
have the best power of detection.  CBS and Multi have less adequate
control over the false positive rate, especially when the number of
change-points is large.
The statistics LLR and SaRa
provide strict asymptotic control of
of the false positive error rate, but LLR has less power than the others,
SaRa suffers a severe loss of power when change-points are close together. 
CBS and Multi show expected power advantages/disadvantages, with
CBS performing better in detecting near-by change-points 
of large amplitude and Multi performing better
in detecting distant change-points of small amplitude.
If at both of relatively nearby change-points the mean moves in the same direction,
CBS and Multi often pick an intermediate value and fail to detect the second change-point.
Our thresholding implementation of Multi is based on 
the approximation (2.9) and 
omits the dynamic programming step 
from the algorithm suggested in 
\cite{Mu14}.  We (and others) found that algorithm to perform 
poorly when used with default parameters;
but our analysis shows that when it is calibrated to have a 
false positive rate comparable to the others, it performs competitively.

In view of the increase in false positives for CBS and Multi
and decrease for LLR when the number of change-points is large,
if one's goals are primarily exploratory,
a visual and/or
rough preliminary analysis that gives us some idea of the
number and configuration of change-points may be helpful in
choosing a detection threshold  that brings the false positive
and false negative rates into balance.


Inversion of the log likelihood ratio statistic is used
to obtain approximate confidence regions for the locations of the
change-points or
jointly for the mean values and the locations.  For the latter
case \cite{Mu14} suggested a quite different method,
which amounts to testing whether there is a change
in the hypothesized mean values between any two hypothesized
change-points.  Numerical examples
suggest that for change-points of large amplitude
our methods provide more accurate estimates,
although our asymptotic control of the confidence level deteriorates if the 
sizes of the changes or the distance between consecutive change-points are not 
sufficiently large.

We have studied briefly the problem noted in the literature 
on detection of copy number variations, 
that there are local drifts that can give the
appearance of change-points where there is none. For these problems
all methods appear to suffer some loss of power, but the method LLR,
which uses a local background seems less likely to experience an
increase in false positives than the top down methods CBS and Multi. 

We have provided a brief discussion of detection of simultaneous changes
in a normal mean and variance, but our analysis to date suggests that the
problem is quite complicated and requires additional study.

We have assumed the observations are independent, which appears 
to be the case for the 
problems motivating our study, both from the nature of the experiments 
and from the
data themselves.  This provides a considerable advantage in estimating the
variance, as discussed above.  For weakly dependent data there
are roughly two different alternatives, which we are now studying.
For short range dependence, if 
the distance between change-points is
proportional to the number of observations,
a number of authors (e.g., \cite{Lund16}) have observed that 
weak convergence arguments
can be used to obtain mathematical results similar 
to those given above, expressed in terms of Brownian motion.  Details
involve correcting the (estimate of the) variance for autocorrelation of the individual
observations. 
A second approach is to use a low order autoregressive model,
which allows one to pursue an asymptotic likelihood analysis, similar
to what we have followed in this paper.   
Our preliminary studies
suggest that both approaches are successful under certain conditions
in controlling the 
false positive rate, but lead to a substantial loss of power,
because of biased estimates of the variance and autocorrelation when
there are change-points.
We expect to provide a thorough discussion of these
problems in the future. 

Another  challenging problem is to detect and estimate local signals in 
spatial data (with or without a temporal variable).  For independent
observations on a rectangular grid, our methods  
generalize easily to signals having a rectangular shape with sides
parallel to the sides of the grid.  A natural question is the extent to
which this is adequate for detection of the signals
having the many irregular shapes possible in higher dimensions.  
Another approach would be to smooth the signals, which opens up the
possibility of dealing with a much larger set of shapes.

\section{Appendix}

\subsection{Appendix A}

In this appendix, we prove Theorem \ref{t2}. Theorem \ref{t3} follows from the same arguments and
therefore its proof is omitted.
The claims stated in the proof will be proved below.

\begin{proof}[Proof of Theorem \ref{t2}]
Recall the basic representations immediately below the statement of the theorem.

The constant $C$ will be chosen in Claim \ref{claim2}.
It is straightforward to verify that the remainder $R$ is of smaller order than $p$:
\bes{
|R|&\leq  \sum_{0\leq i<j<k\leq m \atop j-i, k-j\geq c_0 b^2} \P(Z_{i,j,k}\geq b+1)
+\sum_{0\leq i<j<k\leq m \atop i\leq C\log b \ \text{or}\ k\geq m-C\log b} \P (Z_{i,j,k}\geq b)\\
&\leq m^3 (1-\Phi(b+1))+2C(\log b) m^2 (1-\Phi(b))  \\
&\asymp b^5\varphi(b)e^{-b} +(\log b) b^3 \varphi(b) =o(p).\\
}

For $c_0 b^2\leq u,v\leq m$, define
\be{
d_1=d_1(u, v)=\frac{b}{2u\sqrt{\frac{1}{u}+\frac{1}{v}}},
}
\be{
d_2=d_2(u, v)=\frac{b}{2}\sqrt{\frac{1}{u}+\frac{1}{v}},
}
and
\be{
d_3=d_3(u, v)=\frac{b}{2v\sqrt{\frac{1}{u}+\frac{1}{v}}}.
}
Since $u, v\asymp b^2$, we have $d_i, \nu(2d_i)\asymp 1$ for $i=1,2,3$.

To prove \eq{2.1*},
we only need to show that for any $C \log b\leq i<j<k\leq m- C \log b$ such that $ j-i, k-j\geq c_0 b^2$, we have
\ben{\label{18}
{\p}(\max_{0\leq r<s<t \leq m \atop s-r, t-s \geq c_0 b^2} Z_{r,s,t}\leq b | Z_{i,j,k}=b)
\sim \prod_{i=1}^3 (2d_i^2) \nu(2d_i),
}
where $d_i$ is defined above with $u=j-i, v=k-j$.

In the following we fix any $i,j,k$ such that $C \log b\leq i<j<k\leq m - C \log b$ and $ j-i, k-j\geq c_0 b^2$ and prove \eq{18}. Let $u=j-i, v=k-j$, and let $d_1, d_2, d_3$ be as above.
We also assume the mean $\mu=0$ without loss of generality.

\begin{claim}\label{claim2}
There exists a large enough constant $C$ such that
\be{
{\p} (\max_{0\leq r<s<t \leq m \atop s-r, t-s \geq c_0 b^2, (|r-i| \vee |s-j| \vee |t-k|)\geq C \log b } Z_{r,s,t}> b | Z_{i,j,k}=b)=o(1).
}
\end{claim}

From Claim \ref{claim2}, the maximum in \eq{18} can be restricted to those $r,s,t$ such that
$|r-i|, |s-j|, |t-k|\leq C\log b$.

Next, we note that given $Z_{i,j,k}=b$, except for a set of vanishingly small probability, 
\ben{\label{21}
\frac{S_{j}-S_i}{j-i}= 2d_1(1+o(1))  \ \text{and}\ \frac{S_k-S_j}{k-j}=- 2d_3(1+o(1)).
}
This, together with Theorem 1.6 of \cite{DiFr88} and the fact that $\log b\ll c_0 b^2$, implies that given $Z_{i,j,k}=b$,
\be{
X_{i+1},\dots X_{i+C\log b}, X_{j-C\log b+1}, \dots, X_j, X_{j+1},\dots X_{j+C\log b}, X_{k-C\log b+1}, \dots, X_k
}
are asymptotically mutually independent Gaussian variables with variance 1, the first half of the $X$'s have mean $2d_1$ and the second half of the $X$'s have mean $-2d_3$.
Let us first consider the case $r=i, s=j$ and $k<t\leq k+C \log b$ in \eq{18}.
Note that $Z_{i,j,k}=b$ and $Z_{i,j,t}\leq b$ are equivalent to
\ben{\label{19}
v(S_j-S_i)-u(S_k-S_j)=b u v \sqrt{\frac{1}{u}+\frac{1}{v}}
}
and
\ben{\label{20}
(v+t-k) (S_j-S_i)-u(S_t-S_k+S_k-S_j)\leq b u (v+t-k)\sqrt{\frac{1}{u}+\frac{1}{v+t-k}}.
}
Subtracting \eq{19} from \eq{20} and using Taylor's expansion, we have
\bes{
&(t-k) (S_j-S_i)-u(S_t-S_k)\\
\leq & bu(v+t-k)\sqrt{\frac{1}{u}+\frac{1}{u+t-k}}-buv\sqrt{\frac{1}{u}+\frac{1}{v}}\\
=& (t-k) \Big[ \frac{b}{\sqrt{\frac{1}{u}+\frac{1}{v}}} +\frac{bu}{2v\sqrt{\frac{1}{u}+\frac{1}{v}}}  \Big] (1+o(1));
}
hence, by \eq{21}, given $Z_{i,j,k}=b$, $Z_{i,j,t}\leq b$ is equivalent to
\be{
-(S_t-S_k)+2d_1 (t-k) (1+o(1))\leq [2d_1+d_3](t-k)(1+o(1)).
}
Therefore, with $l:=t-k$,
\bes{
&{\p}(\max_{k<t\leq k+C\log b} Z_{i,j,t}\leq b | Z_{i,j,k}=b)\\
\sim & {\p}\Big(\max_{1\leq l\leq C\log b}\Big\{ \sum_{p=1}^l \big[ -X_{k+p}-d_3(1+o(1)) \big] \Big\}\leq 0 \Big) .
}
Note that $X_{k+p}: p\geq 1$ are i.i.d. $\sim N(0,1)$.
Using the union bound and the facts that $\log b\to \infty$, $d_3\asymp 1$, we have
\bes{
&{\p}\Big(\max_{l > C\log b}\Big\{ \sum_{p=1}^l \big[ -X_{k+p}-d_3(1+o(1)) \big] \Big\}> 0 \Big) \\
\leq &\sum_{l>C\log b}{\p}\Big( \sum_{p=1}^l \big[ -X_{k+p}-d_3(1+o(1)) \big] > 0 \Big) \\
\leq & \sum_{l>C\log b} \exp (- d_3^2(1+o(1))l/2)  = o(1) . 
}
Therefore,
\bes{
& {\p}\Big(\max_{1\leq l\leq C\log b}\Big\{ \sum_{p=1}^l \big[ -X_{k+p}-d_3(1+o(1)) \big] \Big\}\leq 0 \Big)  \\
= & {\p}\Big(\max_{l\geq 1}\Big\{ \sum_{p=1}^l \big[ -X_{k+p}-d_3(1+o(1)) \big] \Big\}\leq 0 \Big)+o(1)\\
= &\sqrt{2}d_3 \nu^{1/2}(2d_3) + o(1),
}
where the last equation is by Corollary 8.44 of \cite{Si85}.

Similar arguments for the other cases show that given $Z_{i,j,k}=b$, the event
\be{
\max_{0\leq r<s<t \leq m \atop |r-i|, |s-j|, |r-k|\leq C \log b} Z_{r,s,t}\leq b
}
is asymptotically the same as the event that six random walks starting from $0$ remain below $0$ 
at all positive times until $C \log b$.
These random walks are asymptotically independent and have independent Gaussian increments 
with variance 1 and means $-d_1, -d_1, -d_2, -d_2, -d_3, -d_3$ respectively.
Therefore,
\be{
{\p}(\max_{0\leq r<s<t \leq m \atop |r-i|, |s-j|, |r-k|\leq C \log b} Z_{r,s,t}\leq b | Z_{i,j,k}=b)
= \prod_{i=1}^3 (2d_i^2) \nu(2d_i)+o(1).
}
This proves \eq{18}.
Now that we have proved \eq{2.1*}, \eq{2.1} follows by the following claim and then by letting $c_0\to 0$.
\begin{claim}\label{claim1}
We have
\be{
\p (\max_{0\leq i < j < k\leq m: \atop j-i, k-j\geq c_0 b^2} |Z_{i,j,k}| \geq b)
\sim 2 \p (\max_{0\leq i < j < k\leq m: \atop j-i, k-j\geq c_0 b^2} Z_{i,j,k} \geq b).
}
\end{claim}
\end{proof}

\begin{proof}[Proof of Claim \ref{claim2}]
We use the union bound 
\bes{
& \p (\max_{0\leq r<s<t \leq m \atop s-r, t-s \geq c_0 b^2, (|r-i| \vee |s-j| \vee |t-k|)\geq C \log b } Z_{r,s,t}> b | Z_{i,j,k}=b)\\
\leq & \sum_{0\leq r<s<t \leq m \atop s-r, t-s \geq c_0 b^2, (|r-i| \vee |s-j| \vee |t-k|)\geq C \log b } \p ( Z_{r,s,t}> b | Z_{i,j,k}=b).
}
The number of terms in the summation is $O(b^6)$.
Note that $Z_{i,j,k}$ and $Z_{r,s,t}$ are both weighted sums of $\asymp b^2$ terms of i.i.d. 
Gaussian variable with weights $\asymp 1/b$ up to sign. If $r,s,t$ are as indicated above in the summation, 
then it is either that $Z_{i,j,k}$ ($Z_{r,s,t}$ resp.) contains at least $C \log b$ terms which are not in 
$Z_{r,s,t}$ ($Z_{i,j,k}$ resp.), or for at least $C \log b$ terms, the weights have opposite sign in $Z_{i,j,k}$ and 
$Z_{r,s,t}$. Therefore, their correlation is at most $1-c\log b/b^2$ where $c$ can be chosen as a universal positive constant when $C$ is larger than some fixed constant. 
Therefore, each conditional probability in the summation is bounded by
\be{
C_1 \exp(-c_2 C \log b)=C_1 b^{-c_2 C},
}
where $C_1$ and $c_2$ are positive constants.
Hence, the summation tends to 0 by choosing a large enough $C$.

\end{proof}

\begin{proof}[Proof of Claim \ref{claim1}]
We write
\bes{
& \p \Big( \max_{0\leq i<j<k \leq m: \atop j-i, k-j\geq c_0 b^2 } |Z_{i,j,k}| \geq b \Big)\\
= &  \p \Big(  \max_{0\leq i<j<k \leq m: \atop j-i, k-j\geq c_0 b^2 } Z_{i,j,k}  \geq b \Big)+ \p \Big(  \max_{0\leq i<j<k \leq m: \atop j-i, k-j\geq c_0 b^2 } \{-Z_{i,j,k}\} \geq b \Big)\\
&- \p \Big(  \max_{0\leq i<j<k \leq m: \atop j-i, k-j\geq c_0 b^2 } Z_{i,j,k} \geq b,
\max_{0\leq r<s<t \leq m: \atop j-i, k-j\geq c_0 b^2 } \{-Z_{r,s,t}\} \geq b \Big)
}
The first two terms are equal by symmetry.
The third term is bounded by
\bes{
&\sum_{0\leq i<j<k \leq m: \atop j-i, k-j\geq c_0 b^2 } \p \Big(   Z_{i,j,k} \geq b,
\max_{0\leq r<s<t \leq m: \atop s-r, t-s\geq c_0 b^2 } \{-Z_{r,s,t}\} \geq b \Big)\\
=&\sum_{0\leq i<j<k \leq m: \atop j-i, k-j\geq c_0 b^2 } \p (Z_{i,j,k}\geq b) \p \Big(\max_{0\leq r<s<t \leq m: \atop s-r, t-s\geq c_0 b^2} \{-Z_{r,s,t}\} \geq b  \Big| Z_{i,j,k}\geq b  \Big).
}
We only need to show that the conditional probability above tends to 0.
We again use the union bound
\bes{
&\p \Big(\max_{0\leq r<s<t \leq m: \atop s-r, t-s\geq c_0 b^2} \{-Z_{r,s,t}\} \geq b  \Big| Z_{i,j,k}\geq b  \Big)\\
\leq & \sum_{0\leq r<s<t \leq m: \atop s-r, t-s\geq c_0 b^2} \p \Big( \{-Z_{r,s,t}\} \geq b  \Big| Z_{i,j,k}\geq b  \Big).
}
There are totally $O(b^6)$ terms in the summation, and each term is subgaussian in $b$. 
Therefore, the conditional probability tends to 0.
\end{proof}

\subsection{Appendix B}

In this appendix, we prove Theorem \ref{t4}. Theorem \ref{t5} follows from the same arguments and
therefore its proof is omitted. The claims stated in the proof will be proved below.

\begin{proof}[Proof of Theorem \ref{t4}]
We use $C$ and $c$ to denote positive constants, which may differ in different expressions.

We denote the probability on the right-hand side of \eq{6} by $p$.
Note that
\be{
p\geq  \p (W_1>a)\geq c e^{-a}.
}
We can decompose $U_{t,\tau,\mu}$ as
\be{
U_{t,\tau,\mu}=\sum_{k=1}^{M+1} (V_{t,k} + Y_k),
}
where
\be{
V_{t,k}:=V_{t,k,\tau}= \frac{(S_{t_k}-S_{t_{k-1}})^2}{2(t_k-t_{k-1})}
- \frac{(S_{\tau_k}-S_{\tau_{k-1}})^2}{2(\tau_k-\tau_{k-1})}
}
and
\be{
Y_k:=Y_{k,\tau,\mu}=\frac{(S_{\tau_k}-S_{\tau_{k-1}}-(\tau_k-\tau_{k-1})\mu_k)^2}{2(\tau_k-\tau_{k-1})}.
}
Given $\tau$ and $\mu$, $\{2Y_k: 1\leq k\leq M+1\}$ are independent and identically distributed $\chi^2(1)$ random variables.
Therefore, we have
\besn{\label{9}
&\p_{\tau, \mu}(\max_{t: |t_k-\tau_k|\leq n_k}U_{t,\mu}>a)\\
=&\int_{0}^{2a}\p_{\tau, \mu}(\max_{t: |t_k-\tau_k|\leq n_k}\sum_{k=1}^{M+1} V_{t,k} >a-\frac{y}{2}|\sum_{k=1}^{M+1} Y_k=\frac{y}{2}) f_{\chi^2_{M+1}}(y) dy\\
&+ \int_{2a}^\infty f_{\chi^2_{M+1}}(y) dy,
}
where $f_{\chi^2_j}(\cdot)$ denotes the density function of a $\chi^2(j)$ random variable.
Under condition \eq{5}, $Y_k\leq a$ implies that
\ben{\label{8}
\frac{S_{\tau_k}-S_{\tau_{k-1}}}{\tau_k-\tau_{k-1}}=\mu_k + o(1),\quad 1\leq k\leq M+1.
}

\begin{claim}\label{cl1}
Conditioning on $\{S_{\tau_k}: 1\leq k\leq M+1\}$ such that \eq{8} is satisfied, we have, with probability $1-o(p)$,
\ben{\label{7}
\frac{S_{t_k}-S_{t_{k-1}}}{t_k-t_{k-1}}=\mu_k + o(1), \quad  1\leq k\leq M+1
}
for all $t$ such that $|t_k-\tau_k|\leq n_k$.
\end{claim}

From Claim \ref{cl1}, in the following we can assume \eq{7}.
For $|t_1-\tau_1|\leq n_1$, we have
\bes{
\frac{S_{t_1}^2}{2t_1}-\frac{S_{\tau_1}^2}{2\tau_1} =& \frac{1}{2} [(S_{t_1}/t_1)(t_1/\tau_1)^{1/2}+S_{\tau_1}/\tau_1][S_{t_1} (\tau_1/t_1)^{1/2}-S_{\tau_1}]\\
=& (\mu_1+o(1))[S_{t_1}-S_{\tau_1}-(t_1-\tau_1)\frac{\mu_1+o(1)}{2}].
}
Similarly, for $2\leq k\leq M$ and $|t_k-\tau_k|\leq n_k$, $|t_{k-1}-\tau_{k-1}|\leq n_{k-1}$,
\bes{
&\frac{(S_{t_k}-S_{t_{k-1}})^2}{2(t_k-t_{k-1})}-\frac{(S_{\tau_k}-S_{\tau_{k-1}})^2}{2(\tau_k-\tau_{k-1})}\\
=&(\mu_k+o(1))[S_{\tau_{k-1}}-S_{t_{k-1}}+(t_{k-1}-\tau_{k-1})\frac{\mu_k+o(1)}{2}]\\
&+ (\mu_k+o(1))[S_{t_k}-S_{\tau_k}-(t_k-\tau_k)\frac{\mu_k+o(1)}{2}],
}
and for $|t_M-\tau_M|\leq n_M$,
\bes{
&\frac{(S_{m}-S_{t_{M}})^2}{2(m-t_{M})}-\frac{(S_{m}-S_{\tau_{M}})^2}{2(m-\tau_{M})}\\
 =&(\mu_{M+1}+o(1))[S_{\tau_M}-S_{t_M}+(t_M-\tau_M)\frac{\mu_{M+1}+o(1)}{2}].
}
Therefore,
\ben{\label{10}
\max_{t: |t_k-\tau_k|\leq n_k}\sum_{k=1}^{M+1} V_{t,k}
=\max_{t: |t_k-\tau_k|\leq n_k}\sum_{k=1}^M (-\delta_k+o(1))[S_{t_k}-S_{\tau_k}-(t_k-\tau_k)\frac{\mu_{k}+\mu_{k+1}+o(1)}{2}].
}

\begin{claim}\label{cl2}
Let $I$ denote the index set $\{1\leq i\leq m:  \tau_k-n_k<i\leq \tau_k+n_k\ \text{for some}\ k=1,\dots, M\}$.
Given $\mu_k'=\mu_k+o(1)$, $1\leq k\leq M+1$, we have, up to an absolute error of $o(p)$ for each of the two probabilities below,
\be{
\p_{\tau, \mu} (  \{X_i: i\in I\}\in A  |\frac{S_{\tau_k}-S_{\tau_{k-1}}}{\tau_k-\tau_{k-1}}=\mu_k', , 1\leq k\leq M+1)\sim  \p (\{X'_i: i\in I\}\in A),
}
where $A$ is an arbitrary Borel set in $\mathbb{R}^{\#\{I\}}$, 
$X$'s are as in the theorem, $X'$'s are independent such that for each $1\leq k\leq M+1$,
$\{X'_i: \tau_{k-1} <i\leq \tau_{k-1}+n_{k-1}\ \&\ i\in I\}$ and $\{X'_i: \tau_k-n_k< i\leq \tau_k\ \&\ i\in I\}$ are identically distributed with distribution $N(\mu_k',1)$.
\end{claim}

Let $\{\xi_j\}_{j\geq 1}$ be independent and identically distributed as $N(0,1)$.
From $n_k \delta_k^2 \gg a$ and $p\geq c e^{-a}$, we have
\be{
\sum_{n>n_k} \p (|\delta_k+o(1)| (\sum_{j=1}^n \xi_j-n |\delta_k+o(1)|/2) \geq a-y/2)\leq C \sum_{n>n_k} e^{-\frac{1}{2}n (\delta_k+o(1))^2}=o(p).
}
From \eq{8}, \eq{10}, Claim \ref{cl2} and the above bound, we have, up to an absolute error of $o(p)$ for each of the two probabilities below,
\be{
\p_{\tau, \mu} (\max_{t: |t_k-\tau_k|\leq n_k}\sum_{k=1}^{M+1} V_{t,k}>a-y/2 | \sum_{k=1}^{M+1} Y_k=\frac{y}{2})
\sim \p (\sum_{k=1}^M \widetilde{W}_k> a-y/2)
}
where $\{\widetilde{W}_k\}_{1\leq k\leq M}$ are independent,
\be{
\widetilde{W}_k=\max \{\widetilde{W}_k^-, \widetilde{W}_k^+\},
}
$\widetilde{W}_k^-$ and $\widetilde{W}_k^+$ are independent and identically distributed, and
\ben{\label{11}
\widetilde{W}_k^-=\sup_{i>0} |\delta_k+o(1)| (\sum_{j=1}^i \xi_j - i |\delta_k+o(1)|/2).
}

We choose $y_1$ and $z$ such that
$1 \ll z \ll \log(y_1)\ll \log\log(a)$.
From (8.49) of Siegmund (1985), we have, for any $z_1\gg 1$,
\be{
{\p}(\widetilde{W}_k>z_1)\sim \p (W_k>z_1).
}
This, together with
Claim \ref{cl3} and Claim \ref{cl4} below, proves the theorem.

\begin{claim}\label{cl3}
We have
\be{
\int_{2b-y_1}^\infty f_{\chi_{M+1}^2} (y)dy=o(p).
}
\end{claim}

\begin{claim}\label{cl4}
We have
\be{
{\p}(\sum_{k=1}^M \widetilde{W}_k>y_1/2)\sim \p (\sum_{k=1}^M \widetilde{W}_k>y_1/2, \min_{1\leq k\leq M}\widetilde{W}_k >z).
}
\end{claim}
\end{proof}

\begin{proof}[Proof of Claim \ref{cl1}]
Assume $k=2$. The other cases follow from the same argument.
It suffices to show that there exists $\epsilon \to 0$ such that
\ben{\label{5.2}
n_1 n_2 {\p} \Big(\frac{|S_{t_2}-S_{t_1}|}{t_2-t_1}\geq \epsilon \Big| S_{\tau_2}-S_{\tau_1}=0 \Big)=o(p)
}
for all $t_1, t_2$ such that $|t_1-\tau_1|\leq n_1$ and $|t_2-\tau_2|\leq n_2$.
Note that conditioning on $S_{\tau_2}-S_{\tau_1}=0$, the mean value of $(S_{t_2}-S_{t_1})/(t_2-t_1)$ is $0$ and by $n_k\ll (m_k \wedge m_{k+1})$, the variance is bounded by $C(n_1+n_2)/m_2^2$.
Therefore,
\besn{\label{5.3}
& \p \Big(\frac{|S_{t_2}-S_{t_1}|}{t_2-t_1}\geq \epsilon \Big| S_{\tau_2}-S_{\tau_1}=0 \Big)\\
\leq & \frac{ C \sqrt{(n_1+n_2)}}{\epsilon m_2} \exp[-\epsilon^2 m_2^2/(2C (n_1+n_2))+\log(n_1)+\log(n_2)].
}
For this to be of smaller order than $p$, we need to choose $\epsilon$ such that
\be{
\frac{\epsilon^2 m_2^2}{2C(n_1+n_2)}-\log(n_1)-\log(n_2)-a\to \infty.
}
Such an $\epsilon\to 0$ exists because $\frac{m_2^2}{n_1+n_2}\gg a$.
\end{proof}

\begin{proof}[Proof of Claim \ref{cl2}]
We only prove that the conditional $X_i$'s can be replaced by the unconditional $X_i'$'s for those $\{\tau_1<i\leq \tau_1+n_1\}$ and $\{\tau_2-n_2<i\leq \tau_2\}$. The other cases follow from the same argument.
Choose $D$ such that
\be{
D^2\ll m_2\ \text{and}\ \frac{D^2}{n_1+n_2}\gg a.
}
Such a $D$ exists because of \eq{5}.
Define
\be{
R_{n_1}'=\sum_{i=\tau_1+1}^{\tau_1+n_1} X_i,\quad  R_{n_2}''=\sum_{i=\tau_2-n_2+1}^{\tau_2} X_i,
}
\be{
\tilde{R}_{n_1}'=\sum_{i=\tau_1+1}^{\tau_1+n_1} X_i',\quad \tilde{R}_{n_2}''=\sum_{i=\tau_2-n_2+1}^{\tau_2} X_i'.
}
By straightforward calculations, we have
\be{
\p_{\tau,\mu} (|R_{n_1}'+R_{n_2}''-\mu'_2(n_1+n_2)|\geq D|\frac{S_{\tau_2}-S_{\tau_{1}}}{\tau_2-\tau_{1}}=\mu_2')\ll p.
}
and
\be{
\p (|\tilde{R}_{n_1}'+\tilde{R}_{n_2}''-\mu'_2(n_1+n_2)|\geq D)\ll p.
}
On the complementary event that $|R_{n_1}'+R_{n_2}''-\mu'_2(n_1+n_2)|< D$ and $|\tilde{R}_{n_1}'+\tilde{R}_{n_2}''-\mu'_2(n_1+n_2)|< D$, 
the conditional density function of $\{X_i: \tau_1<i\leq \tau_1+n_1\ \text{or}\ \tau_2-n_2<i\leq \tau_2\}$ given $(S_{\tau_2}-S_{\tau_1})/(\tau_2-\tau_1)=\mu_2'$ is asymptotically the same as the unconditional density function of $\{X_i': \tau_1<i\leq \tau_1+n_1\ \text{or}\ \tau_2-n_2<i\leq \tau_2\}$ by a straightforward calculation similar to that of \cite{DiFr88}.
\end{proof}

\begin{proof}[Proof of Claim \ref{cl3}]
Note from \eq{51} that the pdf of $W_1$ for large $x$ is larger than that of $\chi_1^2/2$.
Therefore,
\be{
p\geq c \p (\chi^2_{M+2}\ge 2a)\geq c a^{1/2} \p (\chi^2_{M+1}\geq 2a),
}
which is of higher order than
\be{
\p (\chi^2_{M+1}\geq 2a-y_1)
}
by the choice of $y_1$.
\end{proof}

\begin{proof}[Proof of Claim \ref{cl4}]
It suffices to show that for $M\geq 2$,
\be{
{\p} (\min_{1\leq k\leq M}\widetilde{W}_k\leq z | \sum_{k=1}^M\widetilde{W}_k>y_1/2)=o(1).
}
Note that ${\p} (\sum_{k=1}^M\widetilde{W}_k>y_1/2)\asymp y_1^{M-1}e^{-y_1/2}$.
Therefore,
\bes{
&{\p}(\min_{1\leq k\leq M}\widetilde{W}_k\leq z | \sum_{k=1}^M\widetilde{W}_k>y_1/2)\\
\leq & \frac{M {\p}(\sum_{k=1}^{M-1}\widetilde{W}_k>y_1/2-z)}{{\p}(\sum_{k=1}^M\widetilde{W}_k>y_1/2))}
\asymp \frac{M y_1^{M-2} e^{-y_1/2+z}}{y_1^{M-1}e^{-y_1/2}}=o(1)
}
by the choice of $z$.
\end{proof}

\medskip

\bigskip\noindent{\bf Acknowledgement.}  The authors thank Nancy Zhang for 
several helpful discussions and suggestions. We also thank the anonymous referees for their detailed comments which led to many improvements.
XF was partially supported by NUS grant R-155-000-158-112, CUHK direct grant 4053234 and a CUHK start-up grant.
DS was partially supported by the National Science Foundation.

\setlength{\bibsep}{0.5ex}
\def\bibfont{\small}


\end{document}